\documentclass{article}
\usepackage[utf8]{inputenc}

\usepackage[a4paper, margin=1in]{geometry}
\usepackage{amsmath,amsthm,amsfonts,amssymb,amscd,comment,graphicx,svg,notes2bib,mathtools,algorithm,algpseudocode}

\DeclarePairedDelimiter\floor{\lfloor}{\rfloor}

\usepackage{lipsum}
\usepackage[colorlinks]{hyperref}

\newtheorem{cor}{Corollary}
\newtheorem{thm}{Theorem}
\newtheorem{lem}{Lemma}

\newtheorem{definition}{Definition}
\newtheorem{prop}{Proposition}

\newtheorem*{thm*}{Theorem}

\title{Computationally Checking if a Reaction Network is Monotone or Non-expansive}
\author{Alon Duvall}
\date{}

\begin{document}

\maketitle

\begin{abstract}
   We present a systematic procedure for testing whether reaction networks exhibit non-expansivity or monotonicity. This procedure identifies explicit norms under which a network is non-expansive or cones for which the system is monotone—or provides proof that no such structures exist. Our approach reproduces known results, generates novel findings, and demonstrates that certain reaction networks cannot exhibit monotonicity or non-expansivity with respect to any cone or norm. Additionally, we establish a duality relationship which states that if a network is monotone, so is its dual network.
\end{abstract}

\section{Introduction}

Reaction networks, a class of differential equations, are widely used to model biological and chemical phenomena. Understanding their qualitative behavior is crucial, and properties such as monotonicity and contractivity can offer significant insights into stability, convergence, and long-term dynamics.

The study of monotonicity and contractivity has rich theoretical foundations. Monotonicity is well-developed in the context of control theory \cite{1235373} and synthetic biology \cite{10.1371/journal.pcbi.1004881}, while contractivity has found applications in areas like Riemannian geometry \cite{SIMPSONPORCO201474} and data-driven control \cite{hu2024enforcing}. Connections between the two have been explored in \cite{8854175, KAWANO2022105358,jafarpour2023monotonicity}, among others.

Monotonicity and contractivity have been extensively studied in reaction networks. Past works have characterized classes of monotone reaction networks \cite{angeligraphtheoretic, doi:10.1080/14689360802243813, BANAJI20131359, doi:10.1137/120898486}, while others have explored contractivity properties \cite{Electrontransfernetworks, duvall2024interplay} or piecewise-linear Lyapunov functions \cite{7097666, BLANCHINI20142482}. These properties provide powerful tools for understanding global behavior, including stability and convergence. Despite these advances, there remains a need for systematic tools to test whether reaction networks possess a property like monotonicity. 

While algorithmic methods for testing non-expansivity exist, analogous tools for monotonicity are less developed. Addressing this gap, we present a systematic approach \footnote{Code available online at https://github.com/alon-duvall/testing-for-monotonicity-and-non-expansivity-in-reaction-networks} to determine whether a reaction network is monotone with respect to any constant cone. By leveraging geometric insights, our method identifies constraints that these properties impose. As a byproduct, we also develop a new algorithm to test whether a reaction network is non-expansive for any norm. We also establish a duality theorem, demonstrating that a reaction network is monotone with respect to a cone if and only if its dual system is. Using our approach, we resolve an open question from \cite{doi:10.1080/14689360802243813}, proving that a specific reaction network is not monotone with respect to any cone.

\section{Notation}

The following is some basic notation that we make use of. For a vector $x$, the transpose is $x^{\top}$. For a set $U$ the boundary is indicated by $\mbox{bd}(U)$. We define a cone to be any set $K$ such that if $k\in K$ and $\alpha > 0$ then $\alpha k \in K$. Unless stated otherwise, we assume all cones are closed, pointed, and convex. Recall that a cone is \textit{pointed} if $K \cap -K = 0$. For a cone $K$ we indicate the dual cone by $K^*$. Recall for a cone $K$ the dual cone is the set of vectors satisfying $K^* = \{x \in \mathbb{R}^n | x^{\top} k \geq 0 \,\, \forall \,\, k \in K \}$. For a vector $k$ in a cone $K$ we let $k^* = \{x \in K^*| x^{\top} k = 0 \}$. For a cone $K$ we refer to a vector $k$ as an \textit{boundary vector} if $k \in \mbox{bd}(K)$. We say $k$ is an \textit{extreme} vector if it is not strict convex combination of some other vectors in $K$ (i.e., $k \neq \lambda_i k_i$ where $1 = \sum_i \lambda_i$ and for all $i$ $\lambda_i \neq 1$, and $k_i \in K$ is a set of unique vectors). For a vector $x \in \mathbb{R}^n$ we indicate the $i$'th coordinate of $x$ by $(x)_i$. We set $\mathbb{R}_{\geq}^n = \{x \in \mathbb{R}^n| (x)_i \geq 0 \, \forall \, i \}.$ For a set $S$ we indicate its affine span by $\mbox{Aff}(S)$. The dimension of a convex set will be the same as the dimension of its affine span. We indicate the interior of a set $U$ by $\mbox{int}(U)$. For a function $f:\mathbb{R}^n \rightarrow \mathbb{R}$ we indicate its gradient by $\partial f = [\frac{\partial f}{\partial (x)_1},\frac{\partial f}{\partial (x)_2},...,\frac{\partial f}{\partial (x)_n}]$. If $f$ is a function $\mathbb{R}^n \rightarrow \mathbb{R}^n$ we indicate its Jacobian by $J_f$. We say a cone $K$ is generated by some vectors $k_i$ if each vector $k \in K$ can be written as $k = \sum_i \alpha_i k_i$ for some $\alpha_i \geq 0$. We say a cone $K$ is \textit{polyhedral} if it is generated by a finite set of vectors. We indicate the image of a linear map $\Gamma$ by $\mbox{Im}(\Gamma)$ (we sometimes think of a matrix $\Gamma$ as a linear map sending $x$ to $\Gamma x$). For two sets $A,B$ we indicate the \textit{set difference} by $A-B = \{x| \exists a \in A, \exists b \in B \, \, \mbox{such that} \,\, x = a-b \}$. We say a graph is \textit{connected} if between any two nodes of the graph there exists a path (directed or otherwise). We say a directed graph is \textit{strongly connected} if between any two nodes there exists a directed path. A \textit{strongly connected component} of a directed graph is a subgraph $G'$ of a graph $G$ such that $N$ is strongly connected, and it is not a subgraph of any other strongly connected graph $G''$ such that $G' \subseteq G'' \subseteq G$. For a dynamical system $\dot{x} = f(x)$ we indicate the forward time evolution mapping by $\phi_t(.)$

\section{Background}

A \textit{reaction network} is a parameterized family of dynamical systems. A reaction network comprises a set of \textit{species}, which we identify with coordinates in $\mathbb{R}^n$ (for $n$ species). It also consists of a set of reactions, represented as two nonnegative vectors in $\mathbb{R}^n$: one for the \textit{reactant complex} and one for the \textit{product complex}. Each reaction is classified as either \textit{reversible} or \textit{irreversible}. For example, if the reactant complex is $[a_1, a_2, \dots, a_n]$ and the product complex is $[b_1, b_2, \dots, b_n]$, and the species are denoted $A_i$, an irreversible reaction can be represented graphically as:
\[
\sum_i a_i A_i \Rightarrow \sum_i b_i A_i.
\]
For a reversible reaction, we use:
\[
\sum_i a_i A_i \Leftrightarrow \sum_i b_i A_i.
\]

We focus on \textit{non-catalytic reactions}, meaning that $a_i b_i = 0$ for all $i$. In this case, if the reaction’s reversibility is known, the reaction can be unambiguously represented by the vector $[b_1, b_2, \dots, b_n] - [a_1, a_2, \dots, a_n]$. The \textit{reactants} of a reaction are indexed by $A \subset \{i\}_{i=1}^n$ such that $a_j > 0$ for $j \in A$, while the \textit{products} are indexed by $j$ where $b_j > 0$. We refer to this difference as the \textit{reaction vector}. The \textit{stoichiometric matrix} $\Gamma$ is a matrix whose columns are the reaction vectors, and we may refer to the reaction network as $\Gamma$. Each column vector of $\Gamma$ is written as $\Gamma_i$.

Each reaction $\Gamma_i$ is assigned a \textit{reaction rate}, denoted $R_i$. This reaction rate is a function $\mathbb{R}_{\geq}^n \rightarrow \mathbb{R}$, which we assume is $C^1$ on the interior of $\mathbb{R}_{\geq}^n$. Reaction rates are governed by the assumptions of \textit{general kinetics}. Assume reactants are indexed by $A \subseteq \{i\}_{i=1}^n$ and products are indexed by the set $B \subseteq \{i\}_{i=1}^n$. For reversible reactions, the assumptions are as follows:

\begin{enumerate}
    \item We have that $(\Gamma_i)_j(\partial R_i(x))_j < 0$ for $x \in \mathbb{R}_{\geq}^n$ such that $(x)_l > 0$ for all $l \in A$ and all $j$ such that $(\Gamma_i)_j < 0$
    \item We have that $(\Gamma_i)_j(\partial R_i(x))_j < 0$ for $x \in \mathbb{R}_{\geq}^n$ such that $(x)_l > 0$ for all $l \in B$ and for all $j$ such that $(\Gamma_i)_j > 0$
    \item If $(\Gamma_i)_j = 0$ then $(\partial R_i(x))_j = 0$ for $x \in \mathbb{R}_{\geq}^n$
    \item If $(\Gamma_i)_j \neq 0$ and $(x)_j = 0$ then $ (R_i(x)) (\Gamma_i)_j \geq 0 $
\end{enumerate} 
For irreversible reactions with reactants from the set $A$, the assumptions are:

\begin{enumerate}
    \item We have that $(\Gamma_i)_j(\partial R_i(x))_j < 0$ for $x \in \mathbb{R}_{\geq}^n$ such that $(x)_l > 0$ for all $l \in A$ and $(\Gamma_i)_j < 0$
    \item If $(\Gamma_i)_j \geq 0$ then $(\partial R_i(x))_j = 0$ for $x \in \mathbb{R}_{\geq}^n$
    \item If $(\Gamma_i)_j < 0$ and $(x)_j = 0$ then $ R_i(x)  = 0 $
\end{enumerate}

These assumptions closely follow the kinetics assumptions described in \cite{7097666} and are similar to \textit{weakly monotonic kinetics} as described in Definition 4.5 of \cite{SHINAR201292}. They formalize intuitive principles: reactions proceed faster with more reactants, reaction rates depend only on the reactants, and a reaction cannot proceed if a reactant is absent.

The species dynamics of a reaction network $\Gamma$ are governed by the equation $\dot{x} = \Gamma R(x)$, defined on the set $\mathbb{R}^n_{\geq}$. The \textit{stoichiometric compatibility class} of a point $x$ is the set $\{x + \mathrm{Im}(\Gamma)\} \cap \mathbb{R}_{\geq}^n$. A network $\Gamma'$ is a subnetwork of $\Gamma$ if every reaction (i.e., column vector) of $\Gamma'$ is also a reaction of $\Gamma$. The \textit{R-graph} of a reaction network is a directed graph where nodes represent reactions, and there is an arrow from $R_i$ to $R_j$ if the species involved in $R_i$ can influence the reaction rate of $R_j$. Only connected R-graphs are considered, as disconnected ones can be treated as separate reaction networks. If an irreversible reaction $\Gamma_i$ follows \textit{power-law kinetics}, its reaction rate is expressed as $R_j = k_i \prod_i (x)_i^{\alpha_i}$, where $\alpha_i > 0$ when the $i$'th index of $\Gamma_j$ is a reactant, and $\alpha_i = 0$ otherwise.

Suppose a dynamical system $\dot{x} = f(x)$ is defined on a convex, forward-invariant set $C \subseteq \mathbb{R}^n$. The system is said to be \textit{monotone} if there exists a partial ordering on $C$ such that for any $x, y \in C$, if $x \leq y$, then $\phi_t(x) \leq \phi_t(y)$ for all $t \geq 0$. A cone $K \subseteq C$ induces a natural partial ordering on $C$ defined by $x \leq y$ if and only if $y - x \in K$. Throughout this paper, we will restrict our focus to partial orderings induced by cones.

A system is \textit{non-expansive} if there exists a norm $| \cdot |$ on $C$ such that for all $t \geq 0$ and all $x, y \in C$ where $x \neq y$, the inequality $ |\phi_t(x) - \phi_t(y)| \leq |x - y| $ holds. Similarly, the system is \textit{weakly contractive} if $|\phi_t(x) - \phi_t(y)| < |x - y|$ for all such $x, y$, and $t$.

This work focuses on testing reaction networks for their monotonicity and non-expansivity properties. Specifically, we address the following questions:
\begin{enumerate}
    \item Does there exist a norm under which the system $\dot{x} = \Gamma R(x)$ is non-expansive, regardless of the choice of general kinetics?
    \item Does there exist a cone $K$ such that the system $\dot{x} = \Gamma R(x)$ is monotone, independent of the choice of general kinetics?
\end{enumerate}

\section{Some preliminary results}

To determine whether a network is monotone we will rely on Proposition 1.5 in \cite{WALCHER2001543}. We have

\begin{thm}
\label{thm:monotone}
    \cite{WALCHER2001543} A system $\dot{x} = f(x)$ with phase space $\mathbb{R}_{\geq}^n$ is monotone with respect to a proper, pointed and convex cone $K$ iff for any $x \in \mathbb{R}_{\geq}^n$ and for all $k_1\in \mbox{bd}(K)$ and $k_2 \in K^*$ such that $k_2^{\top} k_1  = 0$ we have that $k_2^{\top} \mathcal{J}_f k_1 \geq 0$.
\end{thm}

The inequality in Theorem \ref{thm:monotone} is related to the notion of cross-positivity. A matrix $A$ is \textit{cross-positive} for a cone $K$ if it satisfies $k_2^{\top} A k_1 \geq 0$ for all $k_1 \in K$ and $k_2 \in K^*$ such that $k_2^{\top} k_1 = 0$ \cite{a74ff163-269a-3282-9d4a-078934afa3dc}. In other words, Theorem \ref{thm:monotone} asks that $J_f$ be cross-positive with respect to $K$.

By restricting to an invariant subspace, we can relax the requirement that the cone be proper. If a system $\dot{x} = f(x)$ has an invariant subspace $S$, then we define the restricted system to be the space $S$ evolving in time according to $\dot{x} = f(x)$. We only need the cone to be proper in an invariant subspace. In particular, we have:

\begin{cor}
    Suppose there exists a linear space $S$ such that for all $j \in \mathbb{R}^n$ that $\{S+j\} \cap \mathbb{R}^n_{\geq}$ is forward invariant for our system. Suppose we are given a cone $K \subset S$ of the same dimension as $S$ such that for any $x \in \mathbb{R}^n_{\geq}$ and for all $k_1\in K$ and $k_2 \in K^*$ such that $k_2^{\top} k_1  = 0$ we have that $k_2^{\top} \mathcal{J}_f(x) k_1 \geq 0$, then our system restricted to any set of the form $\{S+j\} \cap \mathbb{R}^n_{\geq}$ is monotone with respect to $K$.  
\end{cor}

\begin{proof}
    This is an immediate extension of Theorem \ref{thm:monotone}. We simply need to restrict the system to $S+j$ and apply the previous theorem.
\end{proof}
We can manipulate the expression in this theorem to get a more geometric statement, specifically for reaction networks. First, we will define a few regions:

\begin{definition}
    Given a reversible reaction vector $v \in \mathbb{R}^n$ we define $Q_1(v) = \{x \in \mathbb{R}^n| (x)_i(v)_i \geq 0 \ \forall \ 1 \leq i \leq n \}$. If we have a irreversible reaction vector $v$, then we define $Q_1(v) = \{x \in \mathbb{R}^n| (x)_i(v)_i \geq 0 \ \forall \ 1 \leq i \leq n \mbox{ such that } (v)_i < 0 \}$  We define $Q_1^+(v) = Q_1(v) \setminus Q_1(-v)$, $Q_1^-(v) = Q_1(-v) \setminus Q_1(v)$ and $Q_2(v) = Q_1(v) \cap Q_1(-v)$.
\end{definition}

\begin{lem}
\label{lem:signs on dual faces}
    Suppose we have a reaction network with only one reaction and the corresponding system $\dot{x} = \Gamma_1 R(x)$. Then our network is monotone with respect to a cone $K$ iff 

    \begin{enumerate}
        \item When $k \in Q_1^+(\Gamma_1) \cap \mbox{bd}(K)$ then for all $k' \in k^*$ we have that $ \Gamma_1^{\top} k'  \leq 0$. 
        \item When $k \in Q_1^-(\Gamma_1) \cap \mbox{bd}(K)$ then for all $k' \in k^*$ we have that $ \Gamma_1^{\top} k' \geq 0$.
        \item When $k \not\in Q_1^-(\Gamma_1) \cup Q_1^+(\Gamma_1) \cup Q_2(\Gamma_1)$ and $k \in \mbox{bd}(K)$ then for all $k' \in k^*$ we have that $\Gamma_1^{\top} k' = 0$.
    \end{enumerate}
\end{lem}

\begin{proof}
    Using Theorem \ref{thm:monotone} we have that the system is monotone with respect to $K$ iff $ (\mathcal{J}_f k)^{\top} k' \geq 0$ whenever $ k^{\top} k' = 0$. We have that $f(x) = \Gamma_1 R(x)$ so that
    \[
     (\mathcal{J}_f k)^{\top} k' =   (\Gamma_1  \partial R(x) k)^{\top} k' =  (\Gamma_1^{\top} k') (\partial R(x) k) \geq 0.
    \]
    That the conditions are necessary and sufficient follow immediately from the last expression above. Indeed, when $k \in Q_1^+(\Gamma_1)$ then $\partial R(x) k$ can take on any negative value, and so we must have $\Gamma_1^{\top} k'  \leq 0$. When $k \in Q_1^-(\Gamma_1)$ then $\partial R(x) k \geq 0$ and can take on any positive value, and so $\Gamma_1^{\top} k' \geq 0$. If $k \not\in Q_1^-(\Gamma_1) \cup Q_1^+(\Gamma_1) \cup Q_2(\Gamma_1)$ then $\partial R(x) k$ can take on any real number, and so for the inequality to hold we must have $ \Gamma_1^{\top} k' = 0$. If $k \in Q_2(\Gamma_1)$ then $\partial R(x) k = 0$ and so $(\Gamma_1^{\top} k') (\partial R(x) k) = 0$.
\end{proof}

The following lemma is also present in \cite{doi:10.1080/14689360802243813} as Theorem 6.1. We include a proof for completeness.

\begin{lem}
\label{lem:one_reac_makes_strict_monotone}

    Suppose we are given a reaction network $\mathcal{N}$. We have that the network $\mathcal{N}$ is monotone with respect to a cone $K$ iff each individual reaction is monotone with respect to the cone.
    
\end{lem}

\begin{proof}
    In the case that $f(x) = \Gamma R(x) = \sum_i \Gamma_i R_i$ (where $\Gamma_i$ is the $i$'th column of the stoichiometric matrix referring to the $i$'th reaction) we have that $\mathcal{J}_f = \sum_i \Gamma_i \partial R_i$. Then our monotonicity condition $(\mathcal{J}_f k_1)^{\top} k_2 \geq 0$ becomes
    \[
     (\mathcal{J}_f k_1)^{\top} k_2 = \sum_i ( \Gamma_i^{\top} k_2 ) ( \partial R_i(x) k_1) \geq 0.
    \]
    For sufficiency, note that if each summand is nonnegative (i.e., each reaction is monotone with respect to $K$), the sum is nonnegative (i.e., the whole network is monotone). For necessity, If one reaction is not monotone, we can find a summand which is negative for some choice of $x,k_1,k_2$, and due to our working with general kinetics we can choose the summand to be as negative as we want. Thus, in this case there exists a choice of kinetics for which the whole sum is negative. 
\end{proof}

\begin{lem}
\label{lem:reaction vectors extremal vectors}
    If a reaction vector $\Gamma_1$ is monotone with respect to a cone $K$, and $\Gamma_1 \in K$, it must be a boundary vector. If $K$ is polyhedral, it must be an extreme vector.
\end{lem}

\begin{proof}
    Suppose $\Gamma_1 \in K$, this implies that $ \Gamma_1^{\top} k' \geq 0$ for all $k' \in K^*$ (this is due to the definition of dual cone). Thus, to satisfy $ (\Gamma_1^{\top} k' )  (\partial R(x) k) \geq 0$ we must have that $\partial R(x) k \geq 0$ for all $k$ whenever $\Gamma_1^{\top} k' > 0$. This implies we always have $k \in Q_1(-\Gamma_1)$ whenever $ \Gamma_1^{\top} k' \neq 0$. If $\Gamma_1 \in \mbox{Int}(K)$ then we always have $\Gamma_1^{\top} k' \neq 0$ and so $K \subseteq Q_1(-\Gamma_1)$ contradicting that $\Gamma_1 \in K$. 

    If $K$ is polyhedral, for an extreme vector $k \in K$ such that $k \neq \Gamma_1$ we can find $k' \in k^*$ such that $k'^{\top} \Gamma_1 \neq 0$ (and since $\Gamma_1 \in K$ we must have in this case $k'^{\top} \Gamma_1 > 0$). This implies that $k \in Q_1(-\Gamma_1)$. If $\Gamma_1$ is not any of the extreme vectors, then again we have $K \subseteq Q_1(-\Gamma_1)$, a contradiction. Thus, $\Gamma_1$ is one of the extreme vectors.
\end{proof}

\begin{lem}
\label{lem:subtract_off}
    Suppose we are given a cone $K$ and a reaction network consisting of a single reaction vector $\Gamma_1 \not\in K$. Our network is monotone with respect to $K$ iff the following properties hold: 
    \begin{enumerate}
        
        \item For all boundary vectors $k \not\in  Q_1(\Gamma_1) \cup Q_1(-\Gamma_1)$ we can find two other vectors $k'\in Q_1(\Gamma_1)$ and $k'' \in Q_1(-\Gamma_1)$ and two positive reals $\alpha_1,\alpha_2$ such that $\alpha_1\Gamma_1 = k - k'$ and $\alpha_2\Gamma_1 = k'' - k$ 
        \item For all boundary vectors $k \in Q_1^+(\Gamma_1)$ we can find a vector $k' \in Q_1(-\Gamma_1) \cap K$ and positive real $\alpha$ such that $\alpha\Gamma_1 = k - k'$.
        \item For all boundary vectors $k \in Q_1^-(\Gamma_1)$ we can find a vector $k' \in Q_1(\Gamma_1) \cap K$ and positive real $\alpha$ such that $\alpha\Gamma_1 = k' - k$.
    \end{enumerate}

\end{lem}

\begin{proof}
For the necessity, suppose we have a vector $k \in K \cap Q_1^+(\Gamma_1)$. Consider two points $x,y$ in the relative interior of our stoichiometric class (so that reaction rates are nonzero). Suppose $y-x = k$. Note that under our assumption of general kinetics, the difference between the vector field at these two points $\Gamma R(y) - \Gamma R(x)$ can be chosen to be equal to the vector $-\Gamma_1$ (use power-law kinetics, for example). Thus, for some small set time, $\phi_t(y) - \phi_t(x)$ can be arbitrarily close to $- \alpha \Gamma_1$ for some $\alpha > 0$. Since our cone is closed, it must in fact then contain $k - \alpha \Gamma_1$. We can continue this way to see that we must have $k - \alpha \Gamma_1 \in K$ for all $\alpha$ such that $k - \alpha \Gamma_1 \not\in Q_1(-\Gamma_1)$. The same argument applies for $k \in Q_1^-(\Gamma_1)$ as well as $k \not\in  Q_1(\Gamma_1) \cup Q_1(-\Gamma_1)$.

For the sufficiency assume all three conditions hold. Then for $k \in Q_1^+(\Gamma_1)$ we have that for $k^*$ such that $k^{\top} k^* = 0$ that
\[
\Gamma_1^{\top} k^*  = \left(\frac{k-k'}{\alpha}\right)^{\top} k^*  =  \left(\frac{-k'}{\alpha}\right)^{\top} k^*  \leq 0.
\]
The last inequality is since in general for $k \in K$ and $k^* \in K^*$ that $ k^{\top} k^*  \geq 0$. For $k \in Q_1^-(\Gamma_1)$ the reasoning is the same.

For $k \not\in  Q_1(\Gamma_1) \cup Q_1(-\Gamma_1)$ we can in the same way from $\alpha_1\Gamma_1 = k - k'$ and $\alpha_2\Gamma_1 = k'' - k$ conclude that $\Gamma_1^{\top} k^*  \leq 0 $ and $ \Gamma_1^{\top} k^* \geq 0$, respectively. Thus, $ \Gamma_1^{\top} k^* = 0$.
\end{proof}

In the proof of necessity for Lemma \ref{lem:subtract_off} we essentially argue for some vectors in our cone, we can choose kinetics which forces additional vectors to be in our cone. We will need a lemma for polyhedral cones. This will make sure we only need to perform finitely many computations in constructing our cones.

\begin{lem}
\label{lem:only check extreme vectors}
    If our cone $K$ is polyhedral, we only need to check the conditions in Lemma \ref{lem:subtract_off} for the extreme vectors of the polyhedral cone to verify monotonicity.
\end{lem}

\begin{proof}
    Suppose the conditions of Lemma \ref{lem:subtract_off} holds for the extreme vectors of a polyhedral cone $K$, which we will label as the vectors $k_1,k_2,...,k_N$. For each $k_i$ there exists a real number $\alpha_i$ such that $k_i + \alpha_i \Gamma_1$ is the minimal quantity creating a vector in $Q_1(\Gamma_1)$, as in Lemma \ref{lem:subtract_off} (note that we cannot have any extreme vectors outside of $ Q_1(\Gamma_1) \cup Q_1(-\Gamma_1)$ by condition 3 of \ref{lem:subtract_off}, since any such vector is a strict convex combination of the two vectors produced in Lemma \ref{lem:subtract_off}).

    Take an arbitrary vector $k \in K$ or $Q_1(-\Gamma_1)$. We can find $\lambda_i$ such that $\lambda_i \geq 0$ and 
    \[
    k = \sum_{i=1}^N \lambda_i k_i.
    \]
    If $k \in Q_1^+(\Gamma_1)$, form the sum $ (\sum_{l=1}^N \lambda_{i} \alpha'_{i}) \Gamma_1$ where $\alpha'_{i} = \alpha_{i}$ if $k_i  \in Q_1^+(\Gamma_1)$, otherwise $ \alpha_{i} = 0$.

    We have that
    \[
    k - (\sum_{l=1}^N \lambda_{i} \alpha'_{i}) \Gamma_1 = \sum_{i=1}^N \lambda_i (k_i - \alpha'_{i} \Gamma_1).
    \]

    Note that $k_i - \alpha'_{i} \Gamma_1 \in Q_1(-\Gamma) \cap K$ for all $i$, and thus $k - (\sum_{l=1}^N \lambda_{i} \alpha'_{i}) \Gamma_1 \in Q_1(-\Gamma) \cap K$. The same argument can be applied for $k \in Q_1^-(\Gamma_1)$ or $k \not\in Q_1(\Gamma_i) \cup Q_1(-\Gamma_i)$. Thus, the conditions of Lemma \ref{lem:subtract_off} are satisfied.
\end{proof}

We will need one more lemma about the dimension of $K$. 

\begin{lem}
\label{lem: cone contains set with full dimension}
    Suppose we are given a reaction network $\Gamma$ which has a strongly connected R-graph and is monotone with respect to a cone $K$. Then there exists a set $C \subset K$ such that $\mbox{Im}(\Gamma) \subseteq \mbox{Aff}(C-C)$.
\end{lem}

\begin{proof}
    Take a vector $k \in K$ such that it has some nonzero coordinates that can impact the kinetics of the reaction $\Gamma_i$. Then by Lemma \ref{lem:subtract_off} the vector $k + \alpha_i \Gamma_i$ will also be in our cone for suitable choice of real number $\alpha$. Thus, for every species of $\Gamma_i$ we will have a vector with non-zero coordinate corresponding to that species. We can continue this process with every other reaction whose kinetics is impacted by a species of $R_i$. Since our R-graph is assumed to be strongly connected, eventually we will carry out going from $k'$ to $k' + \alpha_j \Gamma_j$ for all reaction vectors $j$. Taking the starting and ending vector for all these operations, at least one for each reaction, and forming the convex hull gives us our desired set $C$.
\end{proof}

This lemma implies in particular, that we can find $x$ such that the convex set $(x + \mbox{Im}(\Gamma)) \cap K$ has the same dimension as $\mbox{Im}(\Gamma)$.

\subsection{Lifted networks}

\begin{definition}
    Given a reaction network $\Gamma$, a lift of $\Gamma$ is any network $\hat{\Gamma}$ created by appending some finite number of all zero rows to the bottom of $\Gamma$.
\end{definition}

For example, a lift of $\Gamma = \begin{bmatrix} 1 & 2 \\ -1 & -1 \end{bmatrix}$ could be $\hat{\Gamma} = \begin{bmatrix} 1 & 2 \\ -1 & -1 \\ 0 & 0 \\ 0 & 0 \end{bmatrix}$. When investigating if the reaction network $\Gamma$ is monotone with respect to any cone $K$, we might ask if a lift of $\Gamma$ is monotone with respect to any cone instead. Essentially, we allow ourselves to add dummy species in the hope that this will produce some monotone systems. Below, we describe why a lifted system being monotone is, in a certain sense, equivalent to the system being non-expansive.

\begin{lem}
\label{lem:transversal cones bounded section}
    Let $V$ be a linear subspace and suppose $K \cap V = \emptyset$. Then if $x \in \mathbb{R}^n$ is such that $K \cap (x + V) \neq \emptyset$ then $K \cap (x +V)$ must be a bounded set.
\end{lem}

\begin{proof}
     Let $S = K \cap (x + V)$. Find a vector $k$ in the relative interior of $S$. Note we can write $S$ in the form $S = K \cap (k + V)$, since $(k + V) \cap (x + V) \neq \emptyset$ so they are identical. Note $S$ must be a bounded set. Suppose it was not bounded. Let $B_r$ be the ball of radius $r$ (under the Euclidean distance). Since $S$ is unbounded we can find a sequence of vector $\{v_i\}$ in $S$ such that $v_i \not\in k + B_i$. Since our set is convex we can find $0 < \lambda_i < 1$ and integer $I > 1$ such that $\lambda_i v_i + (1-\lambda_i) k \in \partial B_r $ for $i > I$. Repeating this procedure for $v_i$ with $i > I$, we arrive at an infinite sequence of points with a limit point $\bar{v} - v \in \partial(B_r)$. Note that since the sequence is unbounded we in fact must have that vectors of the form $\alpha_r(\bar{v} - v) + v $ are limit points of $B_r$ for all $r > 0$, and thus $\bar{v}-v \in K$ since $K$ is a closed cone. Since $\bar{v},v \in x + V$ we also have that $\bar{v}-v \in V$. This contradicts $K \cap V = \emptyset$.
\end{proof}

\begin{lem}
\label{lem:lift a cone}
    Suppose we have a reaction network $\Gamma$ (which is possibly lifted) which is monotone with respect to a cone $K$ and has a strongly connected R-graph. If $K \cap \mbox{Im}(\Gamma) = \emptyset$, then there exists a norm for which our system is non-expansive.
\end{lem}

\begin{proof}
    By Lemma \ref{lem:transversal cones bounded section} and Lemma \ref{lem: cone contains set with full dimension} we can find $x$ such that $P = K \cap (x +\mbox{Im}(\Gamma))$ is a bounded and convex set with the same dimension as $\mbox{Im}(\Gamma)$. Now form the symmetric set $B = P - P$. We claim our system is non-expansive with respect to the norm induced by $B$ on each stoichiometric compatibility class. Pick any $y \in \mbox{int}(\mathbb{R}_{\geq}^n)$. Pick $\alpha > 0$ small enough such that for each $z \in y + \alpha B$ we can find $w \in  \mbox{Aff}\{y,y + x + \mbox{Im}(\Gamma)\}$ such that $w \leq z$ and $w \leq y$, where the partial ordering is the one induced by $K$. Note also we can always find such an $\alpha$, since we just need to pick it small enough so that we can find a $z \in y + \alpha B$ such that we can find a $w$ with the property that the convex hull of $w+ 2\alpha P$ is contained in $\mbox{int}(\mathbb{R}_{\geq}^n)$. 

    Due to monotonicity, for small time $\phi_t(w) \leq \phi_t(z)$ and $\phi_t(w) \leq \phi_t(y)$. Thus, there exists a translate of $P$, call it $s + P$, such that $\phi_t(z),\phi_t(y) \in s + P$. Let $l = \phi_t(z) - \phi_t(y)$. Find the smallest $\beta > 0$ such that $\beta l + s + P$ has $\phi_t(y) \in \mbox{bd}(\beta l + s + P)$. Note that we will still have $\phi_t(z) \in \beta l + s + P$, and thus in fact $\phi_t(z) \in \phi_t(y) + (\beta l + s + P -\phi_t(y)) \subseteq \phi_t(y) + B$ (since $0 \in \beta l + s + P -\phi_t(y)$). Thus, we can conclude we are non-expansive with respect to the norm induced by $B$.
\end{proof}

Note that the strongly connected R-graph property was used to guarantee that $P$ has the same dimension as $\mbox{Im}(\Gamma)$. If we assume a priori that the dimensions are the same (perhaps it is the only case we are interested in, so we ignore the more singular cases), then the strong connectivity assumption is unnecessary. 

Indeed, one can prove a more general version of Lemma \ref{lem: cone contains set with full dimension}: If there is a point in a figure, which is closed under the operations of Lemma \ref{lem:produce_new_points}, that is nonzero in coordinate $i$, then every reaction whose kinetics can be impacted by species $i$, as well as every reaction reachable through a directed path in the R-graph from one of these reactions, must be contained in the affine span of the figure. We will not use this more general statement here, however, because we are primarily interested in situations where the figure’s dimension is at least that of the stoichiometric subspace.

Lemma \ref{lem:lift a cone} relates directly to a construction in theorem 7 of \cite{mierczyński2012cooperativeirreduciblesystemsordinary}. The construction in \cite{mierczyński2012cooperativeirreduciblesystemsordinary} also demonstrates how one can go from a monotonicity condition to a contractivity condition.

\section{Building a cone from a vector}

Using Lemma \ref{lem:subtract_off}, we can describe a procedure to build up a cone given some starting vector.

\begin{lem}
\label{lem:produce_new_points}
    Suppose we have a network $\Gamma$ that is monotone with respect to a cone $K$, a vector $v \in K$, reaction vectors $\Gamma_i \in \Gamma$, and real numbers $\alpha,\alpha_1,\alpha_2 > 0$. Then for any index $i$, we have the following:

    \begin{enumerate}

        \item Suppose $v \in Q_1^+(\Gamma_i)$ and $\alpha$ is the minimum number satisfying $v- \alpha \Gamma_i \in Q_1(-\Gamma_i)$. Then $v- \alpha \Gamma_i \in K.$
        \item Suppose $v \in Q_1^-(\Gamma_i)$ and $\alpha$ is the minimum number satisfying $v+ \alpha \Gamma_i \in Q_1(\Gamma_i)$. Then $v+ \alpha \Gamma_i \in K.$
        \item Suppose $v \not\in Q_1(\Gamma_i) \cup Q_1(-\Gamma_i)$ and $\alpha_1, \alpha_2$ are the minimum numbers satisfying $v+ \alpha_1 \Gamma_i \in Q_1(\Gamma_i)$ and $v- \alpha_2 \Gamma_i \in Q_1(-\Gamma_i)$, respectively. Then $v+ \alpha_1 \Gamma_i \in K$ and $v- \alpha_2 \Gamma_i \in K$.

    \end{enumerate}
\end{lem}

\begin{proof}
To prove this lemma, we essentially only need to invert Lemma \ref{lem:subtract_off}. By Lemma \ref{lem:subtract_off} if $v \in Q_1^+(\Gamma_i)$ we can find $v' \in Q_1(-\Gamma_i)$ such that $\beta \Gamma_i = v - v'$. Since our cone is convex, if $\alpha \leq \beta $ then $v - \alpha \Gamma_i$ is also in $K$. Thus, since $v' \in Q_1(-\Gamma_i)$ there exists a minimal $\alpha \leq \beta$ such that $v - \alpha \Gamma_i \in K$, and this vector must always be in our cone(given that $v$ is in our cone). The reasoning for $v \in Q_1^-(\Gamma_i)$ is the same.

If $v \not\in Q_1(\Gamma_i) \cup Q_1(-\Gamma_i)$ then again we can use Lemma \ref{lem:subtract_off}, noting we have $\beta_1$ and $\beta_2$ such that $v+ \beta_1 \Gamma_i \in Q_1(\Gamma_i)$ and $v- \beta_2 \Gamma_i \in  Q_1(-\Gamma_i)$. From this we know that the for  the minimal $\alpha_1$ and $\alpha_2$, such that $\alpha_1 \leq \beta_1$ and $\alpha_2 \leq \beta_2$, we must also have that $v+ \alpha_1 \Gamma_i \in Q_1(\Gamma_i)$ and $v- \alpha_2 \Gamma_i \in  Q_1(-\Gamma_i)$.
\end{proof}

For  $v \in Q_1^+(\Gamma_i)$ refer to forming the vector $v- \alpha \Gamma_i \in k \cap Q_1(-\Gamma_i)$ as \textit{operation 1}, for $v \in Q_1^-(\Gamma_i)$ refer to forming the vector $v+ \alpha \Gamma_i \in k \cap Q_1(-\Gamma_i)$ as \textit{operation 2}, and for $v \not\in Q_1(\Gamma_i) \cup Q_1(-\Gamma_i)$ refer to forming the vectors $v+ \alpha_1 \Gamma_i \in k \cap Q_1(\Gamma_i)$ and $v- \alpha_2 \Gamma_i \in k \cap Q_1(-\Gamma_i)$ as \textit{operation 3}. We sometimes indicate the (possibly multivalued) map of operation $i$ by $o_i(k)$. We say a cone $K$ is \textit{closed} under the operations if the newly formed vectors remain in $K$, i.e., for all $i$ and all $k \in K$ we have that $o_i(k) \subseteq K$.

\begin{thm}
\label{cor:monotone iff subtract off}
    A reaction network is monotone with respect to a cone $K$ if and only if $K$ is closed under the operations of Lemma \ref{lem:produce_new_points} (i.e., we have $o_i(k) \subseteq K$ for all $k \in K$ and $i \in \{1,2,3\}$). The reaction network is non-expansive with respect to a norm with a unit ball $B$ if and only if $B$ is closed under the operations of Lemma \ref{lem:produce_new_points}.
\end{thm}

\begin{proof}
    First we will note that by lifting our network, and constructing a cone as in Definition 11 from \cite{duvall2024interplay} so that now we can use Lemma \ref{lem:subtract_off} or Lemma \ref{lem:produce_new_points} on this constructed cone. Note that this is completely equivalent to applying Lemma \ref{lem:subtract_off} or Lemma \ref{lem:produce_new_points} directly on the unit ball $B$. Thus, we can apply our arguments equally to monotonicity and non-expansivity.

    Now note if we have that $o_i(k) \subseteq K$ (or $o_i(b) \subseteq B$ for $b \in B$) then our network must be monotone (respectively, non-expansive) since we satisfying the conditions in Lemma \ref{lem:subtract_off}. Thus, we have sufficiency. For necessity, note that by Lemma \ref{lem:produce_new_points} if the system is monotone (non-expansive), we must have $o_i(k) \subseteq K$ (or $o_i(b) \subseteq B$ for $b \in B$). 
\end{proof}

\section{Strong monotonicity and weak contractivity}

While we are primarily concerned with computing cones for which a given reaction network is monotone, monotonicity is primarily of interest due to the nice dynamical conclusions one can draw from it. In particular, if one is strongly monotone then there are nice convergence properties for trajectories. This motivates this section, where we show how under certain conditions, monotonicity with respect to a cone $K$ in fact implies strong monotonicity with respect to $K$. We will first describe a well known sufficient condition for strong monotonicity.

\begin{thm}
\label{thm:stong monotone strengthened}
\cite{HIRSCH2006239} Suppose we have a system $\dot{x} = f(x)$ and a proper, pointed and convex cone $K$. Suppose also that $k_2^{\top} \mathcal{J}_f k_1 \geq 0$ for all $k_1 \in \partial K$ and $k_2 \in K^*$ such that $k_2^{\top} k_1 = 0$. Suppose for each $k_1 \in \partial K$ we can find a $k_2 \in K^*$ such that $k_2^{\top} k_1 = 0$ and 
\begin{equation}
k_2^{\top} \mathcal{J}_f k_1 > 0 
\end{equation}
Then our system is strongly monotone with respect to $K$.
\end{thm}

Next we prove that the R-graph being strongly connected can be a sufficient condition to guarantee strong monotonicity.

\begin{thm}
\label{thm:strongly connected to strongly monotone}
    Suppose we have a reaction network $\Gamma$ with a strongly connected R-graph. Suppose the network is monotone with respect to a cone $K \subseteq \mbox{Im}(\Gamma)$ or it is non-expansive for some norm. Then it is strongly monotone with respect to $K$ or weakly contracting with respect to the norm, respectively.
\end{thm}

\begin{proof}
    We will prove the monotonicity case; the non-expansive case can be shown in exactly the same manner. Take a vector $k_1 \in \partial K$. Let $L$ be the set of reactions which share a species with $k_1$. We will show that there is at least one reaction in $L$ and a vector $k_2$ (with $k_2^{\top} k_1 = 0$) for which $\Gamma_i^{\top} k_2 > 0$. 
    
    Suppose this was not the case. Consider the boundary of $S = H \cap K$ where $H$ is the supporting hyperplane of $K$ at $k_1$ with normal $k_2$. Note that the boundary must contain all species involved in any reaction in $L$. Since the R-graph is strongly connected, at $k_1$ we have that $L$ must in fact contain all reactions, since any reaction that shares species with a reaction in $L$ must also be in $L$. This is due to Lemma \ref{lem:produce_new_points}, we can continue using the operations from Lemma \ref{lem:produce_new_points} to move in the direction of different reactions. But since we assumed $\Gamma_i^{\top} k_2 = 0$ always, we always stay on $S$. Hence, $S$ has the same dimension as the cone $K$, a contradiction. Thus, there must exist a reaction such that $\Gamma_i^{\top} k_2 > 0$. 
    
    Note that we must also have $\partial R_i k_1 > 0$. We cannot have $\partial R_i k_1 < 0$ anywhere on $S$ since then we could move outside the cone by Lemma \ref{lem:produce_new_points}. Since the R-graph is strongly connected we can use the operations from Lemma \ref{lem:produce_new_points} to arrive at a point $k_1'$ on $S$ for which $\partial R_i k_1' > 0$. This implies at least one point on the boundary of $S$, call such a point $k_1''$, has $ \partial R_i k_1'' > 0$. Since $k_1$ can be written as a convex sum including $k_1''$, we have $\partial R_i k_1 > 0$.
    
    Since $k_2^{\top} \mathcal{J}_f k_1 \geq (\partial R_i k_1)(k_2^{\top} \Gamma_i ) > 0$, we see that the system satisfies the condition in Theorem \ref{thm:strongly connected to strongly monotone} for $K$, and thus the system is strongly monotone with respect to $K$.
\end{proof}

\section{A monotonicity dichotomy}

In this section, suppose we have a reaction network $\Gamma$ (possibly lifted) that is monotone with respect to $K$. We will now demonstrate a certain dichotomy: either the reaction network is non-expansive or $K \cap \mbox{Im}(\Gamma)$ contains a nonzero vector.

\begin{lem}
\label{lem:intersect a cone}
     Suppose $K \cap \mbox{Im}(\Gamma) \neq 0$. Our reaction network is monotone with respect to $K \cap \mbox{Im}(\Gamma)$.
\end{lem}

\begin{proof}
    Define $K' = K \cap \mbox{Im}(\Gamma)$. If $x-y \in K'$ then $\phi_t(x) - \phi_t(y) \in K'$ since we must have $\phi_t(x) - \phi_t(y) \in \mbox{Im}(\Gamma)$ (stoichiometric compatibility class is invariant) and $\phi_t(x) - \phi_t(y) \in K$ (the system is monotone with respect to $K$).
\end{proof}

\begin{thm}
    The network $\Gamma$ is either non-expansive on its stoichiometric compatibility class or it is monotone with respect to the non-trivial cone $K \cap \mbox{Im}(\Gamma) \neq \{0$\}.
\end{thm}

\begin{proof}
    Suppose $K \cap \mbox{Im}(\Gamma) = \{0\}$. Then by Lemma \ref{lem:lift a cone} the system is non-expansive on its stoichiometric compatibility class. Otherwise $\Gamma$ is monotone with respect to the non-trivial cone $K \cap \mbox{Im}(\Gamma)$ by Lemma \ref{lem:intersect a cone}.
\end{proof}

We will consider two cases when checking for cones. Either the cone is contained in the stoichiometric compatibility class, or it is a lifted cone.

\section{Choosing a starting vector}

To apply the operations from Lemma \ref{lem:produce_new_points} we need at least one vector. In this section, we will examine how to pick a starting vector from which we can build a cone using operations 1,2 and 3. First, we have the following convergence lemma.

\begin{lem}
\label{lem:nonexpansive react converges}
    Suppose that we have a reaction network $\Gamma$ which is non-expansive for a polyhedral norm, and that the reaction rates $R(x)$ are analytic. Suppose the system $\dot{x} = \Gamma R(x) + c$ has at least one  equilibrium in a stoichiometric class $x + \mbox{Im}(\Gamma)$. Then the system's trajectories $x + \mbox{Im}(\Gamma)$ converge to the set of equilibria.
\end{lem}

\begin{proof}
    By \cite{9403888} Theorem 21 we have that an analytic vector field with a bounded trajectory which is non-expansive with respect to a polyhedral norm must converge to its set of equilibria. 
\end{proof}

\begin{lem}
\label{lem:concordance and kernel}
    Suppose the system $\dot{x} = \Gamma R(x) + c$, where $c$ is a constant vector, has two equilibria $b$ and $a$. Then if $v = b-a$ there exists a choice of kinetics so that $\Gamma \partial R v = 0$. Conversely, if there exists $v$ and a choice of kinetics such that $\Gamma \partial R v = 0$ then there exists a choice of kinetics and an $\epsilon > 0$ such that if $a$ is an equilibrium in the interior of $\mathbb{R}^n_{\geq}$ then $a + \epsilon v$ is also an equilibrium point.
\end{lem}

\begin{proof}
    If we have a function $f(R_i)$ which depends only on our choice of kinetics, and a number $\epsilon > 0$, we say it has sign region $(0, \epsilon)$, $(-\epsilon, 0)$ or $(-\epsilon, \epsilon)$ if over all possible choices of kinetics we can achieve any positive values at least up to $\epsilon$, negative values at least down to $-\epsilon$, or at least all the values in the interval $(-\epsilon, \epsilon)$, respectively. We say it has radius at least $\epsilon$ if it has a sign region for $\epsilon$. These sign regions are roughly regions of values we can pick by choosing appropriate power law kinetics. Note that we choose these values and kinetics independently for each reaction $i$.

    First note that $R_i(a + \epsilon v) - R_i(a)$ and $(\partial R_i) v$ have the same sign regions. In fact, by some choice of power law kinetics, we can guarantee for all $i$ each of these functions will be in a sign region of at least radius epsilon. Indeed $R_i(a + \epsilon v) - R_i(a)$ can only be positive iff $(\partial R_i) v$ can only be positive, and the same statement holds with `positive' replaced by `negative'.

    Now if the system has an equilibrium at $a$ and $b$, then then we can pick numbers $\epsilon_i$ in the sign regions of $R_i(b) - R_i(a)$ for each $i$ such that $\sum \Gamma_i \epsilon_i = 0$. We can scale these numbers such that $|\epsilon_i| < \epsilon$. Since $(\partial R_i) v $ has the same sign regions we can pick kinetics such that $ (\partial R_i) v = \epsilon_i$ for all $i$, so that again $ \sum \Gamma_i (\partial R_i) v  = \sum \Gamma_i \epsilon_i = 0$.

    Now if there exists $v$ such that $\Gamma \partial R v = 0$, we can argue in a similar manner, i.e., for each $i$ we can pick $\epsilon_i$ in the sign region of $ \partial R v$ such that $\sum_i \Gamma_i \epsilon_i  = 0$, and then we can guarantee we can choose power law kinetics such that $R_i(a + \epsilon v) - R_i(a) = \epsilon_i$ for all $i$.
\end{proof}

Lemma \ref{lem:concordance and kernel} also relates to \cite{doi:10.1137/15M1034441} definition 4.1.6, as well as Theorem 1.4 in \cite{sign_conditions_for_injectivity_of}.

\begin{definition}
    Define $Z_{\Gamma}$ to be the set of points $v \in \mathbb{R}^n$ such that there exists a choice of kinetics so $\Gamma \partial R v = 0.$
\end{definition}

\begin{definition}
  We say a reaction network is \textit{concordant} if the function $\Gamma R $ is injective on each stoichiometric compatibility class no matter the choice of general kinetics.  
\end{definition}

For the next two lemmas, suppose we have a monotone network $\Gamma$ with a non-expansive (with respect to a polyhedral norm) and concordant subnetwork $\Gamma'$ satisfying $\mbox{rank}(\Gamma') = \mbox{rank}(\Gamma) -1$.

\begin{lem}
\label{lem:contracts to mixed region 2}
    Suppose $\Gamma'$ is monotone with respect to a cone $K \subseteq \mbox{Im}(\Gamma)$ such that $v \in K \setminus \mbox{Im}(\Gamma')$. Then the cone must also contain the set of vectors $(v + \mbox{Im}(\Gamma')) \cap Z_{\Gamma'}$.
\end{lem}

\begin{proof}
    Consider the monotone system $\dot{x} = \Gamma' R'(x)  +c$. Using power-law kinetics, and adjusting $c$ as necessary, construct a vector field with equilibrium points $e$ and $e + v'$ where $v' \in (v + \mbox{Im}(\Gamma')) \cap Z_{\Gamma'}$. Note that $e+  v$ must converge to $e +v'$ since the system is analytic and non-expansive (here we apply Lemma \ref{lem:nonexpansive react converges}, which shows that for an analytic and non-expansive system, any bounded trajectories converge to the set of equilibria). Since we assume the system preserved the cone ordering induced by $K$, we have that $e \leq e+v$ implies $e \leq e+ v'$ and so $v' \in K$.
\end{proof}

\begin{cor}
\label{cor:choosing initial vector}
    Suppose $\Gamma$ is monotone with respect to $K \subseteq \mbox{Im}(\Gamma)$ and there exists $v \in Z_{\Gamma} \cap \mbox{Im}(\Gamma) \setminus  \mbox{Im}(\Gamma')$. Then either $v \in K$ or $-v \in K$.
\end{cor}

\begin{proof}
    The cone $K$ must contain a vector $v' \in \mbox{Im}(\Gamma) \setminus  \mbox{Im}(\Gamma')$. There exists a nonzero multiple of $v$ such that $\alpha v \in v' + \mbox{Im}(\Gamma')$ (since $\Gamma'$ is 1 dimension less than $\Gamma$). By Lemma \ref{lem:contracts to mixed region 2} we have that $\alpha v \in K$ as well.
\end{proof}

Thus, to pick a starting point, we will first reduce our network to a minimal monotone but not non-expansive network. Then we will find a non-expansive and concordant network 1 dimension smaller and check if the $Z_{\Gamma} \cap \mbox{Im}(\Gamma) \setminus \mbox{Im}(\Gamma')$ region is nonempty, and in this case one of the vectors in this region will be our starting point. We can also sometimes demonstrate that $Z_{\Gamma} \cap \mbox{Im}(\Gamma) \setminus \mbox{Im}(\Gamma')$ must be nonempty, using some properties of the Jacobian of the reaction network.

\subsection{Existence of starting vectors}

In this section we will give a guarantee that we can find a starting vector using the strategy describe above. We will need to verify that the Jacobian has vectors in its kernel for some choice of kinetics.

\begin{lem}
\label{lem:mixed vector exists}
    The Jacobian of a concordant reaction network $\Gamma \partial R$ has rank at most the dimension of its stoichiometric compatibility class $\mbox{Im}(\Gamma)$. Assuming the stoichiometric class is not all of $\mathbb{R}^n$, then $ \mbox{Ker}(\Gamma \partial R) \cap \mbox{Im}(\Gamma) = \{0\}$ and there exists $v \not\in \mbox{Im}(\Gamma)$ such that $\Gamma \partial R v = 0$.
\end{lem}

\begin{proof}
    The first claim is immediate, since the image of $\Gamma$ is the same dimension as the stoichiometric compatibility class. If a reaction network is concordant, then we must have $\Gamma \partial R v \neq 0$ for all nonzero $v \in \mbox{Im}(\Gamma)$. Thus, $ \mbox{Ker}(\Gamma \partial R) \cap \mbox{Im}(\Gamma) = \{0\}$. Since the image of $\Gamma \partial R$ has dimension less that $n$, there must exist $v \not\in S$ such that $\Gamma \partial R v = 0$.
\end{proof}

\begin{lem}
    Suppose we have a monotone reaction network $\Gamma$ which is not non-expansive and satisfies $\mbox{dim}(\mbox{Im}(\Gamma)) = n$, and that it has a concordant subnetwork $\Gamma'$ non-expansive with respect to a polyhedral norm in $\mathbb{R}^n$, such that $\mbox{dim}(\mbox{Im}(\Gamma')) = n-1$. Then there exists $v \not\in \mbox{Im}(\Gamma')$ such that for any cone $K$ which $\Gamma$ is monotone with respect to we must have $v \in K$. 
\end{lem}

\begin{proof}
    By Lemma \ref{lem:mixed vector exists} we can find a vector $v \in \mbox{Im}(\Gamma) $ such that $v \not\in \mbox{Im} (\Gamma')$ and $\Gamma' \partial R' v = 0$ (so that $v \in Z_{\Gamma'})$. Indeed, since $ \Gamma' \partial R' (\mbox{Im}(\Gamma)) \subseteq \mbox{Im}(\Gamma)$ and so $\mbox{Ker}(\Gamma' \partial R') \cap \mbox{Im}(\Gamma) \neq \{0\}$, we can find such a $v$.
    
    Thus, for any $\alpha > 0$ we can find two equilibria $a$ and $b$ in $\mathbb{R}^n_{\geq}$ such that $b - a = \alpha v$. Now choose $\alpha$ and analytic kinetics (e.g., power-law kinetics) such that $b$ and $a$ are nonzero equilibria of $\dot{x} = \Gamma' R + c$ in the interior of $\mathbb{R}^n_{\geq}$. Now pick two points $x,y$ that are in the stoichiometric compatibility classes of $a,b$, respectively. 
    
    We can select $a,b,x$ and $y$ such that $x \leq y$ with respect to the order induced by $K$. Note since the system is non-expansive and concordant, by Lemma \ref{lem:nonexpansive react converges} the points $x,y$ must converge to $a,b$, respectively. Since we assume our system is monotone, we must have $b-a \in K$ or $v \in K$. 
\end{proof}

The following corollary often comes up in practice. By Theorem \ref{thm:strongly connected to strongly monotone} we immediately have the following corollary:

\begin{cor}
    A system non-expansive with respect to a polyhedral norm with a strongly connected R-graph is weakly contractive.
\end{cor}
Weakly contractive systems can also be shown to be concordant, as we note in the next lemma.
\begin{lem}
    A weakly contractive system is concordant.
\end{lem}

\begin{proof}
     If a reaction network is weakly contractive, then it cannot have two equilibria in the same stoichiometric compatibility class. Indeed, if both $a$ and $b$ are in $\mbox{Im}(\Gamma)$ and are equilibria we must have $\| \phi_t(a) - \phi_t(b) \|$ is constant, contradicting the definition of weakly contractive. Thus, the system must be concordant.
\end{proof}
These results are useful for guaranteeing we can find a choice of starting vector.

For reversible reaction $\Gamma_i$ let their perpendicular hyperplane to this reaction be $H_{i}$. The next proposition demonstrates that in the special case of reversible networks (i.e., every reaction is reversible), we have a simpler way of finding a vector that is guaranteed to be contained in our cone.

\begin{prop}
\label{prop:reversible network mixed vector}
    Suppose we have a reversible network which is monotone with respect to a cone $K$. Then $K$ must non-trivially intersect every nontrivial $\cap_{i \in A} H_{i}$ where $A$ is some subset of $1,2,...,n$ ($n$ is the number of reactions).

\end{prop}

\begin{proof}
    From Lemma \ref{lem:subtract_off} we can conclude that for any vector $v \in K$ and any $i$ the projection of $v$ onto $H_i$ is also in $K$. Indeed, if $ \Gamma_i^{\top} v  > 0$ then $v \in Q_1^+(\Gamma_i)$ or $v \not\in Q_1(\Gamma_i) \cup Q_1(-\Gamma_i)$. If $v \in Q_1^+(\Gamma_i) $ then by Lemma \ref{lem:subtract_off} we can find $v'$ such that $v-v' = \alpha \Gamma_i$. Since the cone is convex every convex combination of $v$ and $v'$ is in our cone. Since $ \Gamma_i^{\top} v  > 0$ and $\Gamma_i^{\top} v' \leq 0$ we must be able to find a convex combination $v''$ of $v$ and $v'$ such that $ \Gamma_i^{\top} v''  = 0$. The vector $v''$ is the projection of $v$. A similar reasoning applies if $v \in Q_1^-(\Gamma_i)$ or $v \not\in Q_1(\Gamma_i) \cup Q_1(-\Gamma_i)$. By \cite{Halperin} Theorem 1, since we can use our projections as many times as we want, and our cone is closed, we in fact must also have the projection onto any intersection of the $H_i$ must also be in our cone.
\end{proof}

Thus, using the previous proposition, we can simply intersect enough hyperplanes until their intersection is 1-dimensional. This will give us our desired vector.

\section{Algorithm termination conditions}

In this section we describe conditions that guarantee a given reaction network is not non-expansive or monotone for any choice of norm or cone, respectively. These conditions will be used to terminate our algorithm, once they are fulfilled.

For the following we will want an observation about checking the non-expansivity of systems. Asking if we can find a norm ball $B \subset \mbox{Im}(\Gamma)$ for which the reaction network is non-expansive is equivalent to asking if a lifted network $\hat{\Gamma}$ is monotone with respect to a "lifted" cone of $B$ (this is by Theorem 4 in \cite{duvall2024interplay}), call it $B'$. Now asking whether $B'$ is closed under the operations of Lemma \ref{lem:produce_new_points} is the same as asking if $B$ is closed under these operations. Thus, to test for non-expansivity, we simply check whether there exists a norm ball $B \subset \mbox{Im}(\Gamma)$ which is closed under the operations.

For arbitrary vector $v$ we refer to \textit{operation 4} as simply producing the vector $-v$ (i.e., $o_4(v) = -v$). For the rest of the section assume we are given a reaction network $\Gamma.$

\begin{lem}
\label{lem:termination condition nonexpansive}
    Suppose we can pick a nonzero vector $v \in \mbox{Im}(\Gamma)$ and apply operations 1-4 to arrive at the vector $\alpha v$ for $\alpha > 1$. Then our reaction network $\Gamma$ is not non-expansive for any norm.
\end{lem}

\begin{proof}

    First note that since norm balls are symmetric, they are in fact closed under operation 4. Also observe that a norm ball $B$ is closed under operations $1-4$ iff $\beta B$ for any $\beta > 0$. This is due to the fact that if an operation produces $f(v)$ from $v$, then it produces $\beta f(v) $ from $\beta v$.

    Now given that upon starting with $v$, we can produce $\alpha v$ from some sequence of operations, then we can repeat the same operations to find that $\alpha^2 v$ is also in our set. Repeating this we find $\alpha^n v$ is in our set for all positive integers $n$. Thus, our set is unbounded, contradicting that a norm ball should be a bounded set.
\end{proof}

Now moving away from norm balls and back to cones, we have the following termination condition.

\begin{lem}
\label{lem:termination condition monotone}

    Suppose given a starting vector $v$, we can apply operations 1-3 to produce a finite set of vectors $\{k_i\}$ that generate a cone $K$. Suppose that $K$ contains one of our reaction vectors $\Gamma_i$ (or, if it reversible, a nonzero multiple of it) such that $\Gamma_i$ is not an extreme vector of the cone. Then $\Gamma$ is not monotone with respect to any cone that contains $v$.

\end{lem}

\begin{proof}

    Note that any cone $K'$ which contains $v$, must also contain $K$, since $K'$ is closed under operations 1-3 by Theorem \ref{cor:monotone iff subtract off}. Thus, if $\Gamma_i$ is not an extreme vector of $K$, then it is not an extreme vector of $K'$ (while still having $\Gamma_i \in K'$). Thus, $\Gamma$ cannot be monotone with respect to $K'$ by Lemma \ref{lem:reaction vectors extremal vectors}, and so $\Gamma$ is not monotone with respect to any cone which contains $v$.

\end{proof}

\subsection{Convergence of the procedure}

Up to this point, it is still possible that our procedure of building a polyhedral figure from our operations continues indefinitely. Here we establish some results that we can guarantee our procedure will converge (in a precise sense) to a cone or norm for which a given reaction network is monotone or non-expansive, respectively.

\begin{lem}
\label{lem:intersect cones valid}
    If a network $\Gamma$ is monotone with respect to a set of cones $\{K_i\}$ (i.e., for each $i$, $\Gamma$ is monotone with respect to $K_i$). Then if $K = \cap_i K_i$ is nonempty, $\Gamma$ is also monotone with respect to $K$.
\end{lem}

\begin{proof}
    Assume $K$ is nonempty and take $k \in K$. By Theorem \ref{cor:monotone iff subtract off} we have that each $K_i$ is closed under the operations of Lemma \ref{lem:produce_new_points}. Hence, any operation we perform on $k$ will produce a new vector $k'$ that again must be in all $K_i$, and so $k' \in K$. Therefore, $K$ is closed under the operations of \ref{lem:produce_new_points}, and thus $\Gamma$ must be monotone with respect to $K$.
\end{proof}

\begin{lem}
\label{lem:increasing set converge hausdorff}
    Suppose we have an increasing sequence of closed convex sets $C_i$ such that $\cup_{i=1}^{\infty} C_i$ is bounded. Then $\cup_{i=1}^{n} C_i$ converges to $C = \overline{\cup_{i=1}^{\infty} C_i}$ in the Hausdorff metric.
\end{lem}

\begin{proof}
    Take a finite set of points $\{f_i\}$ from $C$ such that every point in $C$ is less than $\epsilon$ away from at least one point in the convex closure of $f_i$. Note that every point in $C$ is eventually less than epsilon from a point in $C_i$ for sufficiently large $i$, and so we can pick $i$ large enough so that every point from the finite set $\{f_i\}$ is less than $\epsilon$ from $C_i$. Due to convexity, this implies every point in $C$ is less than distance $\epsilon$ from a point in $C_i$.
\end{proof}

\begin{cor}
    A reaction network $\Gamma$ which is monotone with respect to a cone which contains a vector $v$ has a minimal cone $K$ which it is monotone with respect to, such that $K$ is contained in all other monotone cones that contain $v$.
\end{cor}

\begin{proof}
    Take the intersection of all the cones that contain $v$, this is a nonempty cone and by Lemma \ref{lem:intersect cones valid} we are also monotone with respect to this cone.
\end{proof}

For the rest of this section, whenever we say a sequence of sets is converging we mean it converges with respect to the Hausdorff metric.

\begin{lem}
    Suppose a network $\Gamma$ is monotone with respect to a minimal cone $K$ containing a vector $v$. Suppose we take a bounded cross section $C$ of $K$ using an affine hyperplane $H$. Then the bounded cross sections $C_i = K_i \cap H$ of cones $K_i$ produced by our procedure converge to $C = K \cap H$.
\end{lem}

\begin{proof}
    First we will work with a cross section of the cone $K$. Every iteration of our procedure gives us a new vector in our cone, which corresponds bijectively to a new point in the bounded convex cross section of the cone (all the convex sets we consider next will be in this cross section).

    Suppose our procedure produces an increasing sequence of convex sets $C_i$ (the sequence is increasing since each iteration can only add new points to our convex set). By Lemma \ref{lem:increasing set converge hausdorff} this increasing sequence must converge in Hausdorff distance to $C = \overline{\cup_{i=1}^{\infty} C_i}$.
\end{proof}

\begin{lem}
    Suppose given a starting vector $v$ the procedure generates a sequence of cones $K_i$ with cross sections $C_i$ which converge to $C$. Then the network $\Gamma$ is monotone with respect to the cone $K$ generated by $C$.
\end{lem}

\begin{proof}
    To show that $\Gamma$ is monotone with respect to $K$, we will argue that $C$ is closed under the operations of Lemma \ref{lem:produce_new_points}. Note that for $c \in C_i$ any operation on $c$ will again produce a point in $K_i$, since this is how we built our increasing convex sets. Now for any point $c \in C$ we can find a sequence of $c_i \in C_{k_i}$ such that $c_i$ converges to $c$. Let $f(c)$ be one of our procedures from Lemma \ref{lem:produce_new_points}. Note these procedures are all continuous. Thus, $f(c_i) \in C$ converges to $f(c)$, and since $K$ is closed we must have $f(c) \in K$ as well. Thus, the cone generated by $C$ is a valid cone closed under the operations of Lemma \ref{lem:produce_new_points}. Hence, by Theorem \ref{cor:monotone iff subtract off}, $\Gamma$ is indeed monotone with respect to the cone generated by $C$.
\end{proof}

\section{Procedure for monotonicity and contractivity}

Now we can write down a procedure which checks for the existence of a lifted cone or a cone contained in the stoichiometric compatibility class, for which our network is monotone. 

\begin{enumerate}
    \item Choose a starting vector (once for non-expansivity once for monotonicity)
    \begin{enumerate}
        \item For non-expansivity, pick the vector to simply be any vector including all species (if the R-graph is connected, any choice would work)
        \item For monotonicity, we use Corollary \ref{cor:choosing initial vector}.
        
    \end{enumerate}
   
    \item Construct additional vectors using Lemma \ref{lem:produce_new_points}
    \item Stop if stopping conditions of either Lemma \ref{lem:termination condition monotone} or Lemma \ref{lem:termination condition nonexpansive} are reached
\end{enumerate}

The choice for the non-expansive starting vector is to prevent some possible degenerate balls from appearing. For example, the network $A \Leftrightarrow B$ and $C \Leftrightarrow D$ has $[1,0,0,0]^{\top},[0,1,0,0]^{\top}$ satisfying the operations for Lemma \ref{lem:produce_new_points}. If the R-graph is strongly connected, the choice of initial vector is less critical, since all species eventually show up in the operations. We also note that our procedure for non-expansivity shares similarities with the procedure for polyhedral Lyapunov functions described in \cite{7402687} and \cite{blanchini2021structuralpolyhedralstabilitybiochemical}.

For the following let $f_i(v)$ be operation $i$ performed on point $v$. If operation $i$ does not apply to $v$ then $f_i(v) = v$. We will refer to the assumptions in Lemma \ref{lem:termination condition nonexpansive} as the non-expansive termination conditions, and the assumptions in Lemma \ref{lem:termination condition monotone} as the monotone termination conditions. Since sometimes the procedure may only converge to a desired object, we apply a heuristic in this scenario. For testing for non-expansivity, we will run the procedure a number of times before taking each extreme vertex in the final figure, and replacing it by a point $\epsilon > 0$ away with rational coordinates $n/m$, where $|m|$ is as small as possible. If we have a cone which is converging to the desired figure, we replace each vector by a nearby vector such that the coordinates are rational multiples of each other (again an $\epsilon >0$ away from the true vector). In practice this may take some trial and error to obtain the desired object. Often the procedure terminates in finitely many steps without need for these heuristics.

\begin{algorithm}[H]
\caption{Procedure for non-expansivity}\label{alg:nonexpansive}
\begin{algorithmic}

\Require {Reaction Network $\Gamma$, maxIterations = n}
\State Vertices = $\{v_1\}$ where $v_1$ has all nonzero coordinates for species involved in the reaction network
\State Iterations $= 1$
\While {Vertices not closed under operations 1-4 and Iterations $\leq n$}
\For {each v $\in$ Vertices}
    \State Vertices $\gets$ Vertices $\cup$ $\{f_1(v)\}$
    \State Vertices $\gets$ Vertices $\cup$ $\{f_2(v)\}$
    \State Vertices $\gets$ Vertices $\cup$ $\{f_3(v)\}$
    \State Vertices $\gets$ Vertices $\cup$ $\{f_4(v)\}$
\EndFor

\State Remove vertices in Vertices which are are not extreme points

\If {Non-expansive termination conditions hold}

    \Return(``System is not non-expansive for any norm")
\EndIf

\State Iterations $\gets$ Iterations $+ 1$

\If {Vertices closed under operations 1-4}

    \State {Ball norm found}
    
    \Return Vertices
\EndIf

\EndWhile

\If {Iterations $> $ n}
    \State Vertices = Approximate rational coordinates (Vertices,$M$,$\epsilon$)
\EndIf

\If {Vertices closed under operations 1-4}

    \State {Ball norm found}
    
    \Return Vertices
\EndIf

\end{algorithmic}
\end{algorithm}

\begin{algorithm}[H]
\caption{Approximating rational coordinates}\label{alg:approximate rationals}
\begin{algorithmic}

\Require {Set of vertices $V$, maxDenominator = M, maxDistance = $\epsilon$}
\For {each v $\in V$ }
    \For{$m \in \{1,...,M\}$}
        \If{$\|mv -\floor{mv}\|_2 < \epsilon$}
            \State $V \leftarrow V \setminus v \cup\frac{ \floor{mv}}{m}$
        \EndIf
    \EndFor
\EndFor
\Return $V$
\end{algorithmic}
\end{algorithm}

\begin{algorithm}[H]
\caption{Picking a starting vector}\label{alg:pick starting vector}
\begin{algorithmic}

\Require Reaction Network $\Gamma$

\If {$\Gamma$ is non-expansive}

    \Return {$\Gamma$ is non-expansive}
\EndIf

\State Check for subnetworks $\Gamma'$ and $\Gamma''$ such that $\Gamma'$ is not non-expansive, and $\Gamma''$ is a subnetwork of $\Gamma'$, concordant, non-expansive, and $\mbox{Im}(\Gamma') \setminus \mbox{Im}(\Gamma'') \neq \emptyset$

\Return Any $v \in Z_{\Gamma''} \cap \mbox{Im}(\Gamma') \setminus  \mbox{Im}(\Gamma'')$

\end{algorithmic}
\end{algorithm}

\begin{algorithm}[H]

\caption{Procedure for monotonicity}\label{alg:monotonicity}

\begin{algorithmic}
\Require Reaction Network $\Gamma$, maxIterations = n

\State Pick a starting vector $k$
\State Extremals = $\{k\}$

\While {Extremals not closed under operations 1-3 and Iterations $\leq $ maxIterations}
\For {each v $\in$ Vertices}
    \State Extremals $\gets$ Extremals $\cup$ $\{f_1(v)\}$
    \State Extremals $\gets$ Extremals $\cup$ $\{f_2(v)\}$
    \State Extremals $\gets$ Extremals $\cup$ $\{f_3(v)\}$
\EndFor

\State Remove vectors in Extremals which are are not extreme vectors

\If {Monotone termination conditions hold}

    \Return(``System is not monotone for any norm")
\EndIf

\State Iterations $\gets$ Iterations $+ 1$

\If {Extremals closed under operations 1-3}

    \State {Ball norm found}
    
    \Return Extremals
\EndIf

\EndWhile

\If {Iterations $> $ maxIterations}

    \State Vectors = Approximate rational vectors (Extremals,$M$,$\epsilon$)
\EndIf

\If {Extremals closed under operations 1-3}

    \State {Ball norm found}
    
    \Return Extremals
\EndIf

\end{algorithmic}
\end{algorithm}

In summary, these algorithms provide a constructive strategy for determining whether a given network $\Gamma$ is monotone or non-expansive, and they often either produce a certifying cone or norm or conclude that no such cone/norm exists.

\section{Duality }

Note that since in our procedures we are essentially building figures vertex by vertex, it might be sometimes quicker to build the vertices of the dual figure. For example, while an $n$-cube has $2^n$ vertices, its dual polytope has only $2n$ vertices. For this reason here we prove a theorem regarding a reaction network and its dual. If we are given an irreversible reaction network with stoichiometric matrix $\Gamma$ the \textit{dual network} is the irreversible network with stoichiometric matrix $\Gamma^{\top}$. Note any reaction network can be turned into an equivalent irreversible reaction network, by simply writing any reversible reactions as two irreversible reactions.

\begin{lem}
\label{lem:transpose isomorphism}
    Given a linear map $A: \mathbb{R}^n \rightarrow \mathbb{R}^n$ let $V$ be the image of $A$ and $V^{\top}$ be the image of $A^{\top}$. Then the linear map $\hat{A}:V^{\top} \rightarrow V$, defined by $\hat{A}(v) = Av$ for $v \in V^{\top}$, is an isomorphism.
\end{lem}

\begin{proof}
    Consider $y \in V^{\top}$ such that $Ay = 0$. Then there exists $x$ such that $y = A^{\top} x$. Then we must have $x^{\top} A A^{\top} x = (A^{\top} x)^{\top}(A^{\top} x) = 0$ or $y = A^{\top} x = 0$. Thus, $A$ is injective on $V^{\top}$, and since both have the same dimension the mapping is a bijection, and so also an isomorphism.
\end{proof}

The following lemma is related to theorem 3 \cite{9433449}. This theorem shows that when a reaction network in concentrations admits a polyhedral norm $V(x)$, that the dual reaction network in rates admits the dual polyhedral norm function $V^*(z)$.   

\begin{lem}
\label{lem:dual networks monotone}
    Suppose the network $\Gamma$ is monotone with respect to a cone $K \subseteq \mbox{Im}(\Gamma)$. Then the dual network is monotone with respect to $\Gamma^{\top} K^*$.
\end{lem}

\begin{proof}
    Note that if $v \in K^*$ then it satisfies $ n^{\top} v \geq 0$ for all $n \in K$. We have that for $ \Gamma^{\top} v\in \Gamma^{\top} K^*$, thinking of $\Gamma$ as a map from $\mbox{Im}(\Gamma^{\top})$ to $\mbox{Im}(\Gamma)$ (which, by Lemma \ref{lem:transpose isomorphism} is an isomorphism and so has an inverse), that
    
    \[ (\Gamma^{\top} v)^{\top} (\Gamma)^{-1} n  = v^{\top} \Gamma (\Gamma)^{-1} n  =  v^{\top} n  \geq 0\] 

    Thus, the vectors in $(\Gamma^{\top} K^*)^*$ are exactly of the form $(\Gamma)^{-1} n$ for $n \in K$. Note that the Jacobian for the dual network (recall we make the network into an irreversible network, before taking the dual) is of the form $\Gamma^{\top} (\partial R)^{\top}$. Now our monotonicity condition is of the form 
    \[((\Gamma)^{-1} n)^{\top} \Gamma^{\top} (\partial R)^{\top} \Gamma^{\top} v = n^{\top}  (\partial R)^{\top} \Gamma^{\top} v =  v^{\top}   \Gamma \partial R n \geq 0.\]

    The last inequality is due to the fact $n \in K$ and $v \in K^*$, so this is the inequality for monotonicity with respect to $K$ in our original system (or contractivity).
\end{proof}

The exact same argument can be used to show if $\Gamma$ is non-expansive for some norm ball $B$, then $\Gamma^{\top}$ is non-expansive for $\Gamma^{\top} B^*$, where $B^*$ is the dual figure. Also, while we worked with irreversible reactions, the same argument works if all reaction are reversible. As an example, take a reversible reaction network with stoichiometric matrix 
\[
\begin{bmatrix}
    -1 & 1 & 0 & 1 \\
    1 & 1 & 1 & 0 \\
    0 & 0 & -1 & -1 
\end{bmatrix}
\]
 It is non-expansive with respect to the octahedron with vertices the columns of the following matrix:

 \[
\begin{bmatrix}
    -1 & 1 & 0 & 0 & 0 & 0 \\
    0 & 0 & 1 & -1 & 0 & 0\\
    0 & 0 & 0 & 0  & -1 & 1
\end{bmatrix}
\]
This is dual to a cube with vertices the columns of the following matrix:
 \[
\begin{bmatrix}
    -1 & 1 & -1 & 1 & -1 & 1 & -1 & 1\\
    -1 & -1 & 1 & 1 & -1 & -1 & 1 & 1\\
    -1 & -1 & -1 & -1  & 1 & 1 & 1 & 1
\end{bmatrix}
\]
Now we have that
\[
\begin{bmatrix}
    -1 & 1 & 0  \\
    1 & 1 & 0 \\
    0 & 1 & -1 \\
    1 & 0 & -1
\end{bmatrix}
\begin{bmatrix}
    -1 & 1 & -1 & 1 & -1 & 1 & -1 & 1\\
    -1 & -1 & 1 & 1 & -1 & -1 & 1 & 1\\
    -1 & -1 & -1 & -1 & 1 & 1 & 1 & 1
\end{bmatrix} =\begin{bmatrix}
    0 & -2 & 2 & 0 & 0 & -2 & 2 & 0\\
    -2 & 0 & 0 & 2 & -2 & 0 & 0 & 2\\
    0 & 0 & 2 & 2 & -2 & -2 & 0 & 0 \\
    0 & 2 & 0 & 2 & -2 & 0 & -2 & 0
\end{bmatrix}.
\]
One can verify the dual network is in fact non-expansive for the columns of this last matrix, by verifying the conditions in Lemma \ref{lem:subtract_off}.

\section{Three worked out examples}

Below, we present three examples illustrating different outcomes of our procedure: (1) a system that is not monotone under any cone; (2) a system that fails non-expansivity for any norm yet does have a monotone cone; and (3) a system from \cite{doi:10.1080/14689360802243813} where we recover a known cone. By Theorem \ref{thm:strongly connected to strongly monotone}, since these networks all have strongly connected R-graphs, they are not only monotone with respect to the cones found, but also strongly monotone.

\subsection{Example 1}
Consider the network (taken from \cite{doi:10.1080/14689360802243813} section 8.7) 
\[
A + B \Leftrightarrow C, A \Leftrightarrow B, 2A \Leftrightarrow C.
\]
This network has stoichiometric matrix
\[
\begin{bmatrix}
    -1 & -1 & -2 \\
    -1 & 1 & 0 \\
    1 & 0 & 1
\end{bmatrix}
\]

First we will verify it is not non-expansive for any norm. Since the  We will take our starting vector as $[2,0,1]^\top$ and using Theorem \ref{cor:monotone iff subtract off} (which applies equally well to norm balls). We have the following sequence:

\[
\begin{bmatrix} -2 \\ 0 \\ 1\end{bmatrix} \xrightarrow[]{+ \begin{bmatrix} 2 \\ 2 \\ -2\end{bmatrix}} \begin{bmatrix} 0 \\ 2 \\ -1\end{bmatrix} \xrightarrow[]{+ \begin{bmatrix} -2 \\ 0 \\ 1\end{bmatrix}} \begin{bmatrix} -2 \\ 2 \\ 0\end{bmatrix} \xrightarrow[]{+ \begin{bmatrix} -2 \\ -2 \\ 2\end{bmatrix}} \begin{bmatrix} -4 \\ 0 \\ 2\end{bmatrix} 
\]

Note that $2 * [-2,0,1]^\top = [-4,0,2]^\top $ and thus by the non-expansivity termination condition (Lemma \ref{lem:termination condition nonexpansive}) our system is not non-expansive for any norm (in fact, the subnetwork consisting just of $ [-2,0,1]^\top$ and $[-1,-1,1]^\top$ is also not non-expansive).

Now we will test for monotonicity. This system is 2 dimensional so we want to look for a 1 dimensional concordant and non-expansive subnetwork $\Gamma'$, for example the network with just $[-2,0,1]^{\top}$. We will take as our starting vector $[-1,3,1] \in Z_{\Gamma'}$ (here we use Corollary \ref{cor:choosing initial vector}). Now taking this as our starting point, we have

\[
\begin{bmatrix} -1 \\ 3 \\ -1 \end{bmatrix} \xrightarrow[]{+ \begin{bmatrix} 1 \\ 1 \\ -1\end{bmatrix}} \begin{bmatrix} 0 \\ 4 \\ -2 \end{bmatrix} \xrightarrow[]{+ \begin{bmatrix} 4 \\ -4 \\ 0\end{bmatrix}} \begin{bmatrix} 4 \\ 0 \\ -2 \end{bmatrix}.
\]

Note that $[-1,-1,1] = -1/4 ([4,0,-2] + [0,4,-2])$ (so we terminate by Lemma \ref{lem:termination condition monotone}). Thus, we are not monotone with respect to any cone, answering the question posed in \cite{doi:10.1080/14689360802243813} in section 8.7.

\subsection{Example 2}
Consider the network
\[
A \Leftrightarrow B + D, B \Leftrightarrow C \Leftrightarrow D, C \Leftrightarrow \emptyset.
\]
This reaction network has stoichiometric matrix
\[
\begin{bmatrix}
    -1 & 0 & 0 & 0 \\
    1 & -1 & 0 & 0 \\
    0 & 1 & -1 & -1 \\
    1 & 0 & 1 & 0
\end{bmatrix}
\]

First we will verify this network is not non-expansive for any norm (note that any subnetwork is in fact non-expansive). Taking $[0,0,1,0]$ as our starting point we have

\[
\begin{bmatrix} 0 \\ 0 \\ 1 \\ 0 \end{bmatrix} \xrightarrow[]{+ \begin{bmatrix} 0 \\ 0 \\ -1 \\ 0\end{bmatrix}} \begin{bmatrix} 0 \\ 0 \\ 0 \\ 1 \end{bmatrix} \xrightarrow[]{+ \begin{bmatrix} 1 \\ -1 \\ 0 \\ -1\end{bmatrix}} \begin{bmatrix} 1 \\ -1 \\ 0 \\ 0 \end{bmatrix} \xrightarrow[]{+ \begin{bmatrix} 0 \\ 1 \\ -1 \\ 0\end{bmatrix}} \begin{bmatrix} 1 \\ 0 \\ -1 \\ 0\end{bmatrix} \xrightarrow[]{+ \begin{bmatrix} 0 \\ 0 \\ 1 \\ 0\end{bmatrix}} \begin{bmatrix} 1 \\ 0 \\ 0 \\ 0\end{bmatrix} \xrightarrow[]{+ \begin{bmatrix} -1 \\ 1 \\ 0 \\ 1\end{bmatrix}} \begin{bmatrix} 0 \\ 1 \\ 0 \\ 1\end{bmatrix} \xrightarrow[]{+ \begin{bmatrix} 0 \\ -1 \\ 1 \\ 0\end{bmatrix}} \begin{bmatrix} 0 \\ 0 \\ 1 \\ 1\end{bmatrix} \xrightarrow[]{+ \begin{bmatrix} 0 \\ 0 \\ -1 \\ 1\end{bmatrix}} \begin{bmatrix} 0 \\ 0 \\ 2 \\ 0\end{bmatrix} .
\]

The ending vector is twice the starting vector, so by Lemma \ref{lem:termination condition nonexpansive} we are not non-expansive. Now we will check the network for monotonicity. The matrix has rank 4, so we want to find a 3 dimensional non-expansive and concordant subnetwork $\Gamma'$, such as $[0,-1,1,0],[0,0,-1,1],[0,0,-1,0]$. This network has $[1,0,0,0] \in Z_{\Gamma'}$.

\[
\begin{bmatrix} 1 \\ 0 \\ 0 \\ 0 \end{bmatrix} \xrightarrow[]{+ \begin{bmatrix} -1 \\ 1 \\ 0 \\ 1\end{bmatrix}} \begin{bmatrix} 0 \\ 1 \\ 0 \\ 1 \end{bmatrix} \xrightarrow[]{+ \begin{bmatrix} 0 \\ -1 \\ 1 \\ 0\end{bmatrix}} 
\begin{bmatrix} 0 \\ 0 \\ 1 \\ 1 \end{bmatrix} 
\xrightarrow[]{+ \begin{bmatrix} 0 \\ 0 \\ 1 \\ -1\end{bmatrix}} \begin{bmatrix} 0 \\ 0 \\ 2 \\ 0 \end{bmatrix} \xrightarrow[]{+ \begin{bmatrix} 0 \\ 2 \\ -2 \\ 0\end{bmatrix}} \begin{bmatrix} 0 \\ 2 \\ 0 \\ 0 \end{bmatrix}
\]
\[
\begin{bmatrix} 0 \\ 0 \\ 2 \\ 0 \end{bmatrix} \xrightarrow[]{+ \begin{bmatrix} 0 \\ 0 \\ -2 \\ 2\end{bmatrix}} \begin{bmatrix} 0 \\ 0 \\ 0 \\ 2 \end{bmatrix} \xrightarrow[]{+ \begin{bmatrix} 2 \\ -2 \\ 0 \\ -2\end{bmatrix}} \begin{bmatrix} 2 \\ -2 \\ 0 \\ 0 \end{bmatrix} \xrightarrow[]{+ \begin{bmatrix} 0 \\ 2 \\ -2 \\ 0\end{bmatrix}} \begin{bmatrix} 2 \\ 0 \\ -2 \\ 0 \end{bmatrix} \xrightarrow[]{+ \begin{bmatrix} 0 \\ 0 \\ 2 \\ -2\end{bmatrix}} \begin{bmatrix} 2 \\ 0 \\ 0 \\ -2 \end{bmatrix}
\]

One can verify that the cone generated by (note that all the column vectors appeared in the above computations) 
\[
\begin{bmatrix}
    0 & 0 & 0 & 1 & 2 & 2 & 2 \\
    2 & 0 & 0 & 0 & -2 & 0 & 0 \\
    0 & 2 & 0 & 0 & 0 & -2 & 0 \\
    0 & 0 & 2 & 0 & 0 & 0 & -2
\end{bmatrix}
\]
is closed under operations 1-3. Thus, by Theorem \ref{cor:monotone iff subtract off} our system is in fact monotone with respect to the cone generated by the columns of the above matrix.

\subsection{Example 3}

Consider Example 8.6 from \cite{doi:10.1080/14689360802243813}, which is 
\[
A + B \Leftrightarrow 2C, A \Leftrightarrow C, B \Leftrightarrow C.
\]
This network has stoichiometric matrix
\[
\begin{bmatrix}
    -1 & -1 & 0 \\
    -1 & 0 & -1 \\
    2 & 1 & 1
\end{bmatrix}
\]
Take the reaction vector $[0,-1,1]^{\top}$ as our subnetwork, note that in this case $[-2,1,1]^{\top} \in Z_{\Gamma'} \cap \mbox{Im}(\Gamma)$. We have

\[
\begin{bmatrix} -2 \\ 1 \\ 1 \end{bmatrix} \xrightarrow[]{+ \begin{bmatrix} 2 \\ 0 \\ -2\end{bmatrix}} \begin{bmatrix} 0 \\ 1 \\ -1 \end{bmatrix} \xrightarrow[]{+ \begin{bmatrix} -1 \\ -1 \\ 2\end{bmatrix}} \begin{bmatrix} -1 \\ 0 \\ 1 \end{bmatrix}.
\]

One can verify that the cone generated by $[0,1,-1]^{\top}$ and $[-1,0,1]^{\top}$ is closed under operations 1-3. Using our methods we have recovered the same cone described in \cite{doi:10.1080/14689360802243813}.

\section{Additional examples}

Below, we provide four more examples illustrating how our algorithm detects monotonicity (or non-expansivity) for additional and sometimes larger reaction networks. Some examples also demonstrate the duality relation or how one might systematically search for larger monotone systems.

\subsection{Example 1}
Consider the reaction network consisting of the three reactions
\[
13A \Leftrightarrow 11B + 7 C, 9B \Rightarrow 7A, 2A \Rightarrow B + 2C.
\]
This reaction network has stoichiometric matrix
\[ \Gamma = 
\begin{bmatrix}
    -13 & 7 & -2 \\
    11 & -9 & 1 \\
    7 & 0 & 2
\end{bmatrix}.
\]
It is monotone with respect to the cone generated by the columns of:

\[
\begin{bmatrix}
    -3 & 0 & 1 & 4 \\
    0 & 11 & 0 & -5 \\
    5 & 7 & 0 & 0 \\
\end{bmatrix}.
\]

We will call this cone $K$. We will also demonstrate our duality relationship on this irreversible network. We will separate the reaction $ 13A \Leftrightarrow 11B + 7 C$ into the two irreversible reactions $ 13A \Rightarrow 11B + 7 C$ and $13A \Leftarrow 11B + 7 C$. The stoichiometric matrix of our reaction network is now

\[ \Gamma = 
\begin{bmatrix}
    -13 & 13 & 7 & -2 \\
    11 & -11 & -9 & 1 \\
    7 & -7 & 0 & 2
\end{bmatrix}.
\]

The dual stoichiometric matrix is

\[
\Gamma^t 
= 
\begin{bmatrix}
    -13 & 11 & 7 \\
    13 & -11 & -7 \\
    7   & -9  & 0 \\
    -2  & 1   & 2
\end{bmatrix}.
\]

This matrix corresponds to the irreversible network $13A + 2D \Rightarrow 13B + 7C, 11B + 9C \Rightarrow 11A + D, 7B \Rightarrow 7A + 2D$. By Lemma \ref{lem:dual networks monotone} we conclude that this network is also monotone with respect to a cone. The dual of the cone $K$ is generated by the columns of:

\[
\begin{bmatrix}
0 & 0 & 55 & 5 \\
0 & -7 & -21 & 4 \\
1 & 11 & 33 & 3
\end{bmatrix}
\]

Now computing $\Gamma^t K^*$ we arrive at the cone generated by the columns of

\[
\Gamma^t K^* 
=
\begin{bmatrix}
7 & 0 & -715 & 0 \\
-7 & 0 & 715 & 0 \\
0 & 63 & 574 & -1 \\
2 & 15 & -65 & 0
\end{bmatrix}.
\]

This is a cone for which the dual network is monotone with respect to. Running Algorithm \ref{alg:monotonicity} we also arrive at an alternative cone for which our system is monotone generated by the columns of:

\[
\begin{bmatrix}
-539 & 0    & 0    & 11   & 6 \\
539  & 0    & 0    & -11  & -6 \\
0    & 1    & -49  & 0    & -7 \\
-157 & 0    & -12  & 1    & 0
\end{bmatrix}.
\]

\subsection{Example 2}

Consider the reaction network
\[
A \Rightarrow B + C, B \Rightarrow A + D, C \Rightarrow A + D, B + C + 2D \Rightarrow A.
\]

This reaction network has stoichiometric matrix
\[ \Gamma = 
\begin{bmatrix}
    -1 & 1 & 1 & 1 \\
    1 & -1 & 0 & -1 \\
    1 & 0 & -1 & -1 \\
    0 & 1 & 1 & -2
\end{bmatrix}
\]

By running Algorithm \ref{alg:monotonicity} we see that it is monotone with respect to the cone generated by the columns of:
\[
\begin{bmatrix}
    -2 & -1 & -1 & -1 & 0 & 0 & 0 & 1 \\
    1 & 0 & 0 & 1 & -1 & 0 & 0 & -1 \\
    1 & 0 & 1 & 0 & 0 & -1 & 0 & -1 \\
    0 & -2 & 1 & 1 & -1 & -1 & 2 & 0
\end{bmatrix}
\]

\subsection{Example 3}

Consider the reaction network

\[
A + B \Leftrightarrow C, A \Rightarrow C, B \Rightarrow C, B \Leftrightarrow A, A \Leftrightarrow \emptyset, B \Leftrightarrow \emptyset.
\]
This reaction network has its stoichiometric matrix
\[ \Gamma =
\begin{bmatrix}
    -1 & -1 & 0 & 1 & -1 & 0 \\
    -1 & 0 & -1 & -1 & 0 & -1 \\
    1 & 1 & 1 & 0 & 0 & 0 \\
\end{bmatrix}
\]

By running Algorithm \ref{alg:monotonicity} we see that it is monotone with respect to the cone generated by the columns of:
\[
\begin{bmatrix}
    -1 & 0 & 0 & 1 \\
    0 & -1 & 1 & 0 \\
    0 & 0 & -1 & -1
\end{bmatrix}
\]

We note that this is the same cone as in examples 8.1 and 8.3 of \cite{doi:10.1080/14689360802243813}. In these examples it is shown that the network $A + B \Leftrightarrow C, A \Leftrightarrow B$ is monotone with respect to this cone. Using our algorithm, we are able to automatically grow this network and find a larger monotone network as well.

\subsection{Example 4}

We give one more example where we have used our algorithm to find a monotone system with multiple reactions. We simply start with a system of one reaction such as $A+B \Leftrightarrow C$ and search through all other reactions with coefficients of $-2,-1,0,1$ or $2$, adding it to our system if the new system is still monotone. One such network we found is 

\[
 A + B \Leftrightarrow C, 2A \Rightarrow C, 2A \Rightarrow B, C \Leftrightarrow 2B, 2B \Rightarrow 2A + C, 2B \Rightarrow A + C, A \Leftrightarrow \emptyset, 2B \Leftrightarrow C, B \Rightarrow C
\]
It has stoichiometric matrix
\[ \Gamma =
\begin{bmatrix}
    -1 & -2 & -2 & 0 & 2 & 1 & -1 & 0 & 0  \\
    -1 & 0 & 1 & 2 & -2 & -2 & 0 & -2 & -1 \\
    1 & 1 & 0 & -1 & 1 & 1 & 0 & 1 & 1
\end{bmatrix}
\]

By running Algorithm \ref{alg:monotonicity} we see that it is monotone with respect to the cone generated by the columns of:
\[
\begin{bmatrix}
    -1 & 0 & 0 & 2 \\
    0 & -2 & 1 & 0 \\
    0 & 1 & -1 & -1
\end{bmatrix}
\]

In summary, these additional examples demonstrate how our procedure can handle a range of reaction networks—from small and relatively simple to more complex systems, and even automatically discover new monotone networks by systematically adding reactions.

\section{Conclusion}

We have demonstrated a procedure to computationally generate cones for which a given reaction network is monotone, or to determine no such cone exists. In the course of building the theory for the procedure, we demonstrate how certain monotonicity conditions really imply non-expansivity, how to check for non-expansivity, as well as how to rule out the possibility of a system being non-expansive. We also note a duality relationship, initially motivated by the possibility of speeding up computations.

\section{Acknowledgements}

The author would like to thank Eduardo D. Sontag, M. Ali Al-Radhawi, Murad Banaji, Dhruv D. Jatkar, Aria Masoomi and Francesco Bullo for helpful conversations on the topic. This work was partially supported by grants AFOSR FA9550-21-1-0289 and NSF/DMS-2052455.

\bibliographystyle{unsrt}
\bibliography{mybib}

\begin{thebibliography}{10}

\bibitem{1235373}
D.~Angeli and E.D. Sontag.
\newblock Monotone control systems.
\newblock {\em IEEE Transactions on Automatic Control}, 48(10):1684--1698, 2003.

\bibitem{10.1371/journal.pcbi.1004881}
Evgeni~V. Nikolaev and Eduardo~D. Sontag.
\newblock Quorum-sensing synchronization of synthetic toggle switches: A design based on monotone dynamical systems theory.
\newblock {\em PLOS Computational Biology}, 12(4):1--33, 04 2016.

\bibitem{SIMPSONPORCO201474}
John~W. Simpson-Porco and Francesco Bullo.
\newblock Contraction theory on riemannian manifolds.
\newblock {\em Systems \& Control Letters}, 65:74--80, 2014.

\bibitem{hu2024enforcing}
Zhongjie Hu, Claudio~De Persis, and Pietro Tesi.
\newblock Enforcing contraction via data.
\newblock {\em arxiv preprint arXiv:2401.07819}, 2024.

\bibitem{8854175}
Yu~Kawano, Bart Besselink, and Ming Cao.
\newblock Contraction analysis of monotone systems via separable functions.
\newblock {\em IEEE Transactions on Automatic Control}, 65(8):3486--3501, 2020.

\bibitem{KAWANO2022105358}
Yu~Kawano and Ming Cao.
\newblock Contraction analysis of virtually positive systems.
\newblock {\em Systems \& Control Letters}, 168:105358, 2022.

\bibitem{jafarpour2023monotonicity}
Saber Jafarpour and Samuel Coogan.
\newblock Monotonicity and contraction on polyhedral cones.
\newblock {\em arXiv preprint arXiv:2210.11576}, 2023.

\bibitem{angeligraphtheoretic}
David Angeli, Patrick De~Leenheer, and Eduardo Sontag.
\newblock Graph-theoretic characterizations of monotonicity of chemical networks in reaction coordinates.
\newblock {\em Journal of Mathematical Biology}, 61(4):581--616, 2010.

\bibitem{doi:10.1080/14689360802243813}
Murad Banaji.
\newblock Monotonicity in chemical reaction systems.
\newblock {\em Dynamical Systems}, 24(1):1--30, 2009.

\bibitem{BANAJI20131359}
Murad Banaji and Janusz Mierczyński.
\newblock Global convergence in systems of differential equations arising from chemical reaction networks.
\newblock {\em Journal of Differential Equations}, 254(3):1359--1374, 2013.

\bibitem{doi:10.1137/120898486}
Pete Donnell and Murad Banaji.
\newblock Local and global stability of equilibria for a class of chemical reaction networks.
\newblock {\em SIAM Journal on Applied Dynamical Systems}, 12(2):899--920, 2013.

\bibitem{Electrontransfernetworks}
Murad Banaji and Stephen Baigent.
\newblock Electron transfer networks.
\newblock {\em Journal of Mathematical Chemistry}, 43(4):1355--1370, 2008.

\bibitem{duvall2024interplay}
Alon Duvall, M.~Ali Al-Radhawi, Dhruv~D. Jatkar, and Eduardo Sontag.
\newblock Interplay between contractivity and monotonicity for reaction networks.
\newblock {\em arxiv preprint arXiv:2404.18734}, 2024.

\bibitem{7097666}
Muhammad Ali Al-Radhawi and David Angeli.
\newblock New approach to the stability of chemical reaction networks: Piecewise linear in rates lyapunov functions.
\newblock {\em IEEE Transactions on Automatic Control}, 61(1):76--89, 2016.

\bibitem{BLANCHINI20142482}
Franco Blanchini and Giulia Giordano.
\newblock Piecewise-linear lyapunov functions for structural stability of biochemical networks.
\newblock {\em Automatica}, 50(10):2482--2493, 2014.

\bibitem{SHINAR201292}
Guy Shinar and Martin Feinberg.
\newblock Concordant chemical reaction networks.
\newblock {\em Mathematical Biosciences}, 240(2):92--113, 2012.

\bibitem{WALCHER2001543}
Sebastian Walcher.
\newblock On cooperative systems with respect to arbitrary orderings.
\newblock {\em Journal of Mathematical Analysis and Applications}, 263(2):543--554, 2001.

\bibitem{a74ff163-269a-3282-9d4a-078934afa3dc}
Hans Schneider and Mathukumalli Vidyasagar.
\newblock Cross-positive matrices.
\newblock {\em SIAM Journal on Numerical Analysis}, 7(4):508--519, 1970.

\bibitem{mierczyński2012cooperativeirreduciblesystemsordinary}
Janusz Mierczyński.
\newblock Cooperative irreducible systems of ordinary differential equations with first integral.
\newblock {\em arXiv preprint arXiv:1208.4697}, 2012.

\bibitem{HIRSCH2006239}
Morris~W. Hirsch and Hal Smith.
\newblock Chapter 4 {M}onotone {D}ynamical {S}ystems.
\newblock In A.~Cañada, P.~Drábek, and A.~Fonda, editors, {\em Handbook of Differential Equations: Ordinary Differential Equations Volume 2}, volume~2 of {\em Handbook of Differential Equations: Ordinary Differential Equations}, pages 239--357. North-Holland, 2006.

\bibitem{9403888}
Saber Jafarpour, Pedro Cisneros-Velarde, and Francesco Bullo.
\newblock Weak and semi-contraction for network systems and diffusively coupled oscillators.
\newblock {\em IEEE Transactions on Automatic Control}, 67(3):1285--1300, 2022.

\bibitem{doi:10.1137/15M1034441}
Murad Banaji and Casian Pantea.
\newblock Some results on injectivity and multistationarity in chemical reaction networks.
\newblock {\em SIAM Journal on Applied Dynamical Systems}, 15(2):807--869, 2016.

\bibitem{sign_conditions_for_injectivity_of}
Stefan M{\"u}ller, Elisenda Feliu, Georg Regensburger, Carsten Conradi, Anne Shiu, and Alicia Dickenstein.
\newblock Sign conditions for injectivity of generalized polynomial maps with applications to chemical reaction networks and real algebraic geometry.
\newblock {\em Foundations of Computational Mathematics}, 16(1):69--97, 2016.

\bibitem{Halperin}
Israel Halperin.
\newblock The product of projection operators.
\newblock {\em Acta Sci. Math. (Szeged)}, 23(2):96--99, 1962.

\bibitem{7402687}
Franco Blanchini and Giulia Giordano.
\newblock Polyhedral lyapunov functions for structural stability of biochemical systems in concentration and reaction coordinates.
\newblock In {\em 2015 54th IEEE Conference on Decision and Control (CDC)}, pages 3110--3115, 2015.

\bibitem{blanchini2021structuralpolyhedralstabilitybiochemical}
Franco Blanchini, Carlos~Andrés Devia, and Giulia Giordano.
\newblock Structural polyhedral stability of a biochemical network is equivalent to finiteness of the associated generalised petri net.
\newblock {\em arxiv preprint arXiv:2109.01709}, 2021.

\bibitem{9433449}
Franco Blanchini and Giulia Giordano.
\newblock Dual chemical reaction networks and implications for lyapunov-based structural stability.
\newblock {\em IEEE Control Systems Letters}, 6:488--493, 2022.

\end{thebibliography}

\end{document}